\documentclass[final]{amsart}

\usepackage{amssymb,amsthm,amsmath,amsfonts,amsrefs}
\usepackage{mathrsfs,mathtools}     
\usepackage{latexsym}
\usepackage{verbatim}
\usepackage{enumitem}               
\usepackage{eucal}                  
\usepackage[english]{babel}         
\usepackage{comment}
\usepackage{hyperref}

\newenvironment{claim}[1]{\vspace{1mm}\textbf{Claim #1}:}{\vspace{1mm}}

\newtheorem{thm}{Theorem}[section]
\newtheorem{lem}[thm]{Lemma}
\newtheorem{cor}[thm]{Corollary}

\theoremstyle{definition}
\newtheorem{definition}[thm]{Definition}
\newtheorem{example}[thm]{Example}

\theoremstyle{remark}
\newtheorem{rem}[thm]{Remark}

\newcommand{\R}{\mathbb{R}}
\newcommand{\Z}{\mathbb{Z}}
\newcommand{\N}{\mathbb{N}}
\newcommand{\C}{\mathbb{C}}

\newcommand{\sigmap}{ \sigma_{\text{p}} }
\newcommand{\co}{ \textnormal{co-\!} }
\newcommand{\ess}{ \sigma_{\text{ess}} }
\newcommand{\ind}{\text{ind}}
\newcommand{\nul}{ \textnormal{N} }
\newcommand{\range}{ \textnormal{R} }
\newcommand{\domain}{ \textnormal{D} }
\newcommand{\divergence}{ \text{div} }

\begin{document}

\title{Properties of differential operators with vanishing coefficients}

\author{Daniel Jordon}
\email{djordon@umich.edu}
\address{Transportation Research Institute, University of Michigan, Ann Arbor, MI 48109}

\maketitle

\begin{abstract}
In this paper, we investigate the properties of linear operators defined on $L^p(\Omega)$ that are the composition of differential operators with functions that vanish on the boundary $\partial \Omega$. We focus on bounded domains $\Omega \subset \R^d$ with Lipshitz continuous boundary. In this setting we are able to characterize the spectral and Fredholm properties of a large class of such operators. This includes operators of the form $Lu = \divergence( \Phi \nabla u)$ where $\Phi$ is a matrix valued function that vanishes on the boundary, as well as operators of the form $Lu = D^{\alpha} (\varphi u)$ or $L = \varphi D^{\alpha} u$ for some function $\varphi \in \mathscr{C}^1(\bar{\Omega})$ that vanishes on $\partial \Omega$.
\end{abstract}

\section{Introduction}



In this article we study the properties of linear operators when we allow the leading coefficient functions to vanish on the boundary of the domain. For example, the differential equation: 
\begin{align}
\label{some elliptic operator}
L u = -\divergence( \Phi \nabla u ) = f,
\end{align}
where $f \in L^p(\Omega)$, and $\Phi(x) \in \C^{d\times d}$ has been extensively studied when $\Phi$ is uniformly positive definite on $\bar{\Omega}$. The operator $L$ is called uniformly elliptic. For more on such operators, see \cites{demengel2012functional, gilbarg1977elliptic, han2011elliptic, krylov1996lectures, maz2011sobolev} and the references therein.
%
Less is known when the uniform positivity assumption on $\Phi$ is relaxed. In \cite[\S 6.6]{gilbarg1977elliptic}, Trudinger and Gilberg partially relax the condition. In particular, they assume $\Phi \in \mathscr{C}^{0,\gamma}(\bar{\Omega})$ for some $\gamma \in (0,1)$, and if $x_0 \in \partial \Omega$ then there exists a suitably chosen $y \in \R^d$ such that $\Phi(x_0) \cdot (x_0-y) \neq 0$. With this restriction, they establish existence and uniqueness of solutions to \eqref{some elliptic operator}. 
In \cite{murthy1968boundary}, Murthy and Stampacchia studied the properties of weak solutions to \eqref{some elliptic operator} in the case where there exists a positive function $m$ with $m^{-1} \in L^p(\Omega)$ such that
\begin{align} \label{generic restriction}
{v} \cdot \Phi(x) {v} \geq m(x) |{v}|^2,
\qquad \text{for a.e. } x\in \Omega, \text{ and all } {v} \in \R^d.
\end{align}
Studies on the properties of solutions to $Lu = f$, when $L$ is non-uniformly elliptic, can be found in \cites{trudinger1971non-uniformly, trudinger1981harnack, trudinger1977maximum}, as well as \cite{coffman1976bvp}, and \cite{franchi1998irregular}, where the authors assume restrictions on $\Phi$ that are similar to \eqref{generic restriction}. The results presented here address the Fredholm properties of $L$ in the case when $\Phi = 0$ on $\partial \Omega$, and/or when $v \cdot \Phi(x) v \geq m(x) |v|$ for some positive function $m \in \mathscr{C}^{1}(\bar{\Omega})$ with $m^{-1} \not\in L^p(\Omega)$.

Other examples of a differential equation with vanishing coefficients arise when studying linear stability of solutions to non-linear PDE. The operator
\begin{align}\label{linearized compacton equation}
Lu = (\varphi u)_{xxx} + (\varphi u)_{x} + b u_x,
\end{align}
defined on $L^p(-1, 1)$ where $\varphi(x) = a \cos^2(\tfrac{\pi}{2} x)$ and $a,b \in \R$, arose when studying compactly supported solutions to
\begin{align*}
u_t = (u^2)_{xxx} + (u^2)_x.
\end{align*}

Our results can be used to accomplish two goals. The first is assessing the solvability of the boundary value problem $Lu = f$ where $u = g$ on the boundary. To that end, we analyze the Fredholm properties of $L$. Our second goal is establishing well-posedness (or ill-posedness) of the Cauchy problem $u_t = Lu$, where $u = g$ when $t = 0$. For this goal, we present results on the spectrum of $L$.

The operators studied here are linear differential operators on $L^p(\Omega)$, elliptic or otherwise, that have coefficient functions on the leading order derivative term that vanish on $\partial \Omega$. For the matrix valued function $\Phi : \bar{\Omega} \to \C^{d\times d}$, we only require that at least one eigenvalue vanishes on $\partial \Omega$. The operators shown in \eqref{some elliptic operator} and \eqref{linearized compacton equation} are examples of operators that can be analyzed using the results presented here.

\section{The main results} \label{section:main results}

\subsection{Preliminaries}

Throughout this article we make the following assumptions:
\begin{itemize}
\item The domain $\Omega \subset \R^d$ is open and bounded and $\partial \Omega$ is Lipshitz continuous.
\item The function $\varphi : \bar{\Omega} \to \R$ is such that $\varphi \in \mathscr{C}^k(\bar{\Omega})$ for some positive integer $k$, $\varphi > 0$ on $\Omega \subset \R^d$, and $\ker \varphi = \partial \Omega$.
\item The ambient function space for the differential operator $L$ is $L^p(\Omega)$ where $1 < p < \infty$.
\end{itemize}
We say the scalar valued function $\varphi$ is \emph{simply vanishing on} $\partial \Omega$ if for each $y \in \partial \Omega$ there exists an $a \neq 0$ such that
\begin{align}\label{vanishing_definition intro}
	\lim_{x \to y} \frac{\varphi(x)}{\text{dist}(x,\partial \Omega)} = a,
\end{align}
where the limit is taken in $\Omega$. For the matrix valued function $\Phi : \bar{\Omega} \to \C^{d\times d}$ we make restrictions on its eigenfunctions, defined as the functions $\varphi_i$ such that $\Phi(x) v = \varphi_i(x)v$ for some $v \in \C^d$. We say the matrix valued function $\Phi : \bar{\Omega} \to \C^{d\times d}$ is \emph{simply vanishing on} $\partial \Omega$ if $\Phi(x)$ is positive semi-definite in $\bar{\Omega}$ and for each fixed $i$, the eigenfunction $\varphi_i$ is either strictly positive on $\bar{\Omega}$ or simply vanishing on $\partial \Omega$, with at least one $i$ such that $\varphi_i$ is simply vanishing on $\partial \Omega$.




\subsection{Results}

In this section, we summarize the results proved in this article. In the following theorem, $\lfloor a \rfloor$ denotes the integer part of $a$.

\begin{thm}
\label{thm:kinda hardy intro}
Let $\Omega \subset \R^d$ be a bounded open set and let $\varphi \in \mathscr{C}^1(\bar{\Omega})$ be simply vanishing on $\partial \Omega$. Fix $m \in \N$ and $1 < p < \infty$. Assume $k \in \N$ is such that 
\begin{align*}
k > \tfrac{d}{p} + (m-1)\big\lfloor \tfrac{d}{p} \big\rfloor,
\end{align*}
and that the boundary $\partial \Omega$ is $\mathscr{C}^k$. If $u \in L^{p}(\Omega)$ and $\varphi^m u \in W^{k+m,p}(\Omega)$ then
\begin{align*}
u \in W^{\kappa,p}(\Omega), \quad \text{where } \kappa \coloneqq k- m\big\lfloor \tfrac{d}{p} \big\rfloor,
\end{align*}
and there exists a $c > 0$, independent of $u$, such that
\begin{align*}
\|u\|_{W^{\kappa, p}(\Omega)} \leq c \|\varphi^{m} u\|_{W^{k+m,p}(\Omega)}.
\end{align*}
\end{thm}

The above result is proven in section~\ref{subsection:domain_vanishing} as Theorem~\ref{thm:kinda hardy}, with the estimate proven in Remark~\ref{rem:implicit bound}. 
%
%
%
The following theorem is proven in section~\ref{subsection:compactness} as Theorem~\ref{thm:not closed but closable}.

\begin{thm}\label{thm:not closed but closable intro}
Let $\Omega \subset \R^d$ be an open and bounded set with $\mathscr{C}^{0,1}$ boundary. Let $\varphi \in \mathscr{C}^{1}(\bar{\Omega})$ be simply vanishing on $\partial \Omega$. If $A$ is Fredholm on $L^p(\Omega)$ with domain $W^{k,p}(\Omega)$ for some $k > 0$, then $\varphi^m A$, $m \geq 1$, is not closed on its natural domain,
\begin{align*}
\domain(\varphi^m A) = \{ u \in L^p : u \in \domain(A), \; \varphi^m Au \in L^{p}(\Omega) \} \equiv \domain(A),
\end{align*}
but $\varphi^m A$ is closable.
\end{thm}

The above theorem tells us that even simple operators do not have the desirable property of being closed on their `natural domain'. For example, the operator $L u = \sin(\pi x) u_{xx}$ is not closed on $W^{2,p}(0,1)$ for any $p \in (1,\infty)$ by Theorem~\ref{thm:not closed but closable intro}. We can use Theorem~\ref{thm:kinda hardy intro} to get an estimate on the properties of the domain, as demonstrated in the following example.

\begin{example}\label{exa:example_for_domain}
Let $\Omega = (0,1)$ and consider the operator $L$ acting on $L^p(\Omega)$ defined by $Lu = \varphi u_{xxx}$ where $\varphi$ is simply vanishing. The operator $L$ is of the form $\varphi A$ where $A$ is Fredholm. By Theorem~\ref{thm:not closed but closable intro}, $L$ is not closed on its natural domain, $W^{3,p}(\Omega)$, but is closable. Let $\bar{L}$ denote the closure of $L$, and let $\varphi^{(k)} : L^p(\Omega) \to W^{3-k,p}(\Omega)$ denote the multiplication operator $u \mapsto \varphi^{(k)} u$ where the function $\varphi^{(k)}$ denotes the $k$-th derivative of the function $\varphi$, and as an abuse of notation we set $\varphi = \varphi^{(0)}$. After rewriting $L$ as
\begin{align*}
L u = (\varphi u)_{xxx} - 3 (\varphi^{(1)} u )_{xx} + 3 (\varphi ^{(2)} u)_x - \varphi^{(3)} u, 
\end{align*}
we see that 
\begin{align*}
\domain(\bar{L}) = \domain( A^3 \varphi ) \cap \domain( A^2 \varphi^{(1)} ) \cap \domain( A \varphi^{(2)} ) \cap \domain( \varphi^{(3)} ),
\end{align*}
where $A$ is the derivative operator on $L^p(\Omega)$. Specifically, if $\varphi \in \mathscr{C}^3(\bar{\Omega})$ is simply vanishing, then the fact that $\domain(A^3) = W^{3,p}(\Omega)$ implies $\domain(A^3\varphi) \subset W^{2,p}(\Omega)$ by Theorem~\ref{thm:kinda hardy intro}. By the same theorem, we have $W^{2,p}(\Omega) \subset \domain(A^{k} \varphi^{(3-k)})$ for $k = 0,1,2$ so the best we can do is $\domain(\bar{L}) \subset W^{2,p}(\Omega)$.
More concretely, if we set $\varphi(x) = \sin(\pi x)$ and fix $p=2$, then one can construct functions $u \in \domain(\bar{L})$ such that $u \not\in W^{3,2}(\Omega)$ but $u \in W^{2,2}(\Omega)$. See \cite[\S 3.1]{jordon2013properties}.
\end{example}

The following theorem speaks about the range of the multiplication operator. It 
is proved in section~\ref{subsec:range vanishing} as Theorem~\ref{thm:range closed vanishing}. 

\begin{thm}
\label{thm:range closed vanishing intro}
Let $\Omega \subset \R^d$ be open and bounded with $\mathscr{C}^{0,1}$ boundary and assume the function $\varphi \in \mathscr{C}^k(\bar{\Omega})$ is simply vanishing on $\partial \Omega$. Then the range of the operator $u \mapsto \varphi^m u$ is closed in $W^{k,p}(\Omega)$ whenever $k \geq m$ and is not closed when $k < m$.
\end{thm}

If we know the range of the multiplication operator $u \mapsto \varphi^m u$ is closed in $W^{k,p}(\Omega)$ then we necessarily have
\begin{align*}
\|u\|_{L^p(\Omega)} \leq c \|\varphi^m u\|_{W^{k,p}(\Omega)}
\end{align*}
for some constant $c > 0$. 

The matrix valued analogs of Theorems~\ref{thm:kinda hardy intro} and \ref{thm:range closed vanishing intro} are proved in section~\ref{subsection:matrix_functions}. One implication is illustrated in the following example.

\begin{example}
Let $\Omega\subset \R^d$ be open and bounded with $\mathscr{C}^{0,1}$ boundary. Let $\Phi \in \mathscr{C}^1(\bar{\Omega}; \R^{d\times d})$ be simply vanishing on $\partial \Omega$. Then by the matrix analog of Theorem~\ref{thm:range closed vanishing intro} (which is Theorem~\ref{thm:range closed vanishing_matrices}) we know that the range of $\Phi^m$, $m \in \N$, is closed in 
\begin{align*}
W^{1,2}(\Omega^d) \coloneqq \underbrace{W^{1,2}(\Omega) \times  \cdots \times W^{1,2}(\Omega),}_{d \text{ copies}}
\end{align*}
    if and only if $m = 1$. This implies $\Phi^m : L^2(\Omega^d) \to W^{1,2}(\Omega^d)$ is semi-Fredholm if and only if $m=1$. Now, it is well known that the weak gradient $\nabla : W^{1,2}(\Omega) \subset L^2(\Omega) \to L^2(\Omega^d)$ and its adjoint, $\divergence(\cdot) : L^2(\Omega^d) \to L^2(\Omega)$, are semi-Fredholm. Thus, the non-uniformly elliptic operator,
\begin{align}
Lu = \divergence( \Phi^m \nabla u),
\end{align}
is semi-Fredholm on $L^2(\Omega)$ if and only if $m=1$.
\end{example}


\begin{thm}
\label{thm:roots equal spectrum intro}
Let $\Omega \subset \R^d$ be open and bounded with $\mathscr{C}^{0,1}$ boundary, $m,k \in \N$ with $0 < m < k$, and let $A$ be densely defined on $L^p(\Omega)$. Assume the following:
\begin{itemize}
\item The operator $A$ is closed on $L^p(\Omega)$ and $\domain(A) \subset W^{k,p}(\Omega)$.
\item There exists a $u \in \domain(A)$ such that $u \notin W^{m,p}_0(\Omega) \cap W^{k,p}(\Omega)$. 
\item The resolvent set $\rho(A)$ is non-empty.
\end{itemize}
If $\varphi \in \mathscr{C}^k(\bar{\Omega})$ is simply vanishing then
\begin{align*}
\ess(A \varphi^m ) = \ess(\overline{\varphi^m A^*} ) =\C. 
\end{align*}
Moreover, if either $A\varphi^m$ or $\overline{\varphi^m A^*}$ is Fredholm then
\begin{align*}
\sigmap(\overline{\varphi^m A^*} ) = \C.
\end{align*}
\end{thm}

In the above theorem, $\sigmap(L)$ and $\ess(L)$ denote the point spectrum and essential spectrum of $L$ respectively. The definition of the essential spectrum is given in section~\ref{section:spectrum} and the result is proven as Theorem~\ref{thm:roots equal spectrum}. Its import is demonstrated in the following example.

\begin{example}
This example continues from Example~\ref{exa:example_for_domain}, where $\Omega = (0,1)$, and $L = \varphi u_{xxx}$. For any $u \in W^{3,p}(\Omega)$, there exists $a,b \in \R$ such that $u + ax + b \in W^{1,p}_0(\Omega)$, which implies 
\begin{align*}
W^{3,p}(\Omega) = W^{3,p}(\Omega) \cap W^{1,p}_0(\Omega) \oplus \text{span}\{1,x\}.
\end{align*}
Now consider the multiplication operator $\varphi : L^p(\Omega) \to W^{3,p}(\Omega)$ given by $u \mapsto \varphi u$ where $\varphi$ is simply vanishing on $\partial \Omega$. Then we know that the range of $\varphi$ is $W^{3,p}(\Omega) \cap W^{1,p}_0(\Omega)$, which has co-dimension 2 in $W^{3,p}(\Omega)$. Thus, if we let $A$ denote three applications of the weak derivative operator on $L^p(\Omega)$, then we know that $\rho(A)$ is nonempty, and $\domain(A) = W^{3,p}(\Omega) \subset\subset L^p(\Omega)$ by the Rellich-Kondrachov Theorem. Then we see that $A \varphi$ has finite dimensional nullspace and the range has finite co-dimension. This shows $A \varphi$ is Fredholm. Applying Theorem~\ref{thm:roots equal spectrum intro} yields $\sigmap(\bar{L}) = \sigmap(\overline{\varphi A}) = \C$.

More concretely, we have $\sigmap (\overline{\sin(\pi x) u_{xxx}} ) = \C$. The same holds if we set $\varphi(x) = \sin^2(\pi x)$, but not necessarily when we set $\varphi(x) = \sin^m(\pi x)$ where $m \in \N$ and $m \geq 3$.
\end{example}

\subsection{Outline of the article}
The article is structured as follows:
\begin{itemize}
\item Section~\ref{section:preliminaries} introduces the notation and basic definitions that are used in this article.
\item Section~\ref{section:basics} goes over some basic properties of closed and Fredholm operators.
\item Section~\ref{section:properties_vanishing} covers the properties of the operators $u \mapsto \varphi u$ and $\mathbf{u} \mapsto \Phi \mathbf{u}$. The domain and range of the operator $u \mapsto \varphi u$ is covered in sections~\ref{subsection:domain_vanishing} and \ref{subsec:range vanishing} respectively. Matrix valued operators are handled in section~\ref{subsection:matrix_functions}
\item We study various properties of differential operators composed with vanishing operators in section~\ref{section:differential_with}. In particular, we focus on the Fredholm properties and spectra of the operators $A\varphi$ and $\varphi A$ where $A$ is a Fredholm differential operator.
\end{itemize}

\section{Notation and definitions}\label{section:preliminaries}

We will review some of the basic definitions and introduce the notation used in this article. We will use capital letters, such as $W$, $X$, $Y$, or $Z$, to denote a Banach space. We use $\mathcal{B}(X, Y)$ to denote the set of bounded linear operators from $X$ to $Y$, and $\mathcal{C}(X,Y)$ to denote the set of closed and densely defined linear operators from $X$ to $Y$. The sets $\mathcal{B}(X)$ and $\mathcal{C}(X)$ denote the sets $\mathcal{B}(X,X)$ and $\mathcal{C}(X,X)$ respectively.

The domain and range of a linear operator $A$ will be denoted by $\domain(A)$ and $\range(A)$ respectively, and we use $\nul(A)$ to denote the nullspace of $A$. If $A \in \mathcal{C}(X,Y)$, then $\domain(A)$ equipped with the graph norm,
\begin{align*}
\|x\|_{\domain(A)} \coloneqq \|x\|_{X} + \|Ax\|_{Y}, \qquad x \in \domain(A),
\end{align*}
is a Banach space, and we call $\|\cdot\|_{\domain(A)}$ the $A$\emph{-norm}. When referring to the composition of two linear operators $A$ and $B$, the subspace
\begin{align*}
\domain(AB) = \{ \, x \in \domain(B) : Bx \in \domain(A) \, \},
\end{align*}
is called the \emph{natural domain} of $AB$. 

If $A$ is a linear operator from $X$ to $Y$, any closed operator ${A}_1$ where $\domain(A) \subset \domain(A_1)$ and $A = {A}_1$ on $\domain(A)$ is called a {\it closed extension} of $A$. We call $A$ {\it closable} if there exists a closed extension of $A$. We denote $\bar{A}$ as the closure of $A$, and it is the `smallest' closed extension, in the sense that $\domain(\bar{A}) \subset \domain(A_1)$ for any operator $A_1$ that is a closed extension of $A$. An operator is closable if every sequence $\{x_n\} \subset \domain(A)$ where $x_n \to 0$ in $X$ and $Ax_n \to y$ in $Y$ implies $y = 0$.

A Banach space $Y$ is said to be \emph{continuously embedded} in another Banach space $X$ if there exists an operator $P \in \mathcal{B}(Y,X)$ that is one-to-one. The space $Y$ is said to be \emph{compactly embedded} in $X$ if $P$ is also compact and we write $Y \subset\subset X$ whenever $Y$ is compactly embedded in $X$. For Sobolev spaces, we take $P$ to be the inclusion operator, which we denote as $\iota$. 

Most of the analysis takes place in $L^p(\Omega)$ and the Sobolev spaces $W^{k,p}(\Omega)$, where $k \in \N$, $\Omega \subset \R^d$ is an open and bounded set, and, unless stated otherwise, $1 < p < \infty$. The closure of a set $\Omega \subset \R^d$ will be denoted by $\bar{\Omega}$ and the boundary of $\Omega$ will be denoted by $\partial \Omega$. The space $\mathscr{C}^k(\Omega)$ denotes the space of all functions from $\Omega$ to $\R$ that are $k$-times continuous differentiable everywhere in $\Omega$, and $\mathscr{C}^{k}_{0}(\Omega) \subset \mathscr{C}^k(\Omega)$ denotes the subspace of those functions with compact support in $\Omega$. The space $W^{k,p}_0(\Omega)$ denotes the closure of $\mathscr{C}^{\infty}_{0}(\Omega)$ in $W^{k,p}(\Omega)$. If $u$ is weakly differentiable, we let $D^{\alpha}u$ denote the $\alpha$-th weak derivative of $u$, where $\alpha = (\alpha_1,\ldots,\alpha_d) \in \Z^d_+$ is a multi-index, and we let $|\alpha| = \alpha_1 + \cdots +\alpha_d$ denote the order of $\alpha$. We use $\nabla^{(k)} u$ to denote the vector of all weak derivatives of $u$ with order $k$, and set $\nabla u = \nabla^{(1)} u$ to be the gradient of $u$. 

We say $\partial \Omega$ is $\mathscr{C}^{k,\gamma}$ if for each point $y \in \partial \Omega$, there exists an $r > 0$ and a $\mathscr{C}^{k,\gamma}$ function $\gamma : \R^{d-1} \to \R$ such that
\begin{align*}
\Omega \cap B(y,r) = \{ x \in B(y,r) : x_d > \gamma(x_1,\ldots, x_{d-1})\},
\end{align*}
where $B(x,r) \coloneqq \{ y \in \R^d : | x - y| < r\}$, and $\mathscr{C}^{k,\gamma}$ is a H\"{o}lder space.

\begin{definition}
Let $\Omega \subset \R^d$ be a bounded and open set. For any $\varphi \in \mathscr{C}^1(\bar{\Omega})$, we will call the multiplication operator $u \mapsto \varphi u$ \emph{vanishing} if $\ker \varphi = \partial \Omega$. As an abuse of notation, we will use $\varphi$ to refer to the multiplication operator $u \mapsto \varphi u$. 
\end{definition}

\begin{definition}
Let $\Omega \subset \R^d$ be a bounded and open set. Take $\text{dist}(x,\partial \Omega) \coloneqq \inf_{y \in \partial \Omega} |x - y|$ to be the distance from $x$ to the boundary of $\Omega$. Let $\varphi \in \mathscr{C}^{1}(\bar{\Omega})$. We say $\varphi$ is \emph{simply vanishing on} $\partial \Omega$ if $\ker \varphi = \partial \Omega$, and for each $y \in \partial \Omega$ there exists an $a \neq 0$ such that
\begin{align}\label{vanishing_definition}
	\lim_{x \to y} \frac{\varphi(x)}{\text{dist}(x,\partial \Omega)} = a,
\end{align}
where the limit is taken within $\Omega$. The multiplication operator $u \mapsto \varphi u$ is called simply vanishing on $\partial \Omega$ if the function $\varphi$ is simply vanishing on $\partial \Omega$.
\label{def:simply_vanishing}
\end{definition}

\begin{definition}
Let $\Omega \subset \R^d$ be a bounded and open set. We say the function $\varphi \in \mathscr{C}^{m}(\bar{\Omega})$ is \emph{vanishing of order} $m$ on $\partial \Omega$ if $\ker \varphi = \partial \Omega$, $D^{\alpha} \varphi = 0$ on $\partial \Omega$ when $|\alpha| < m$, and $\nabla^{(m)} \varphi \neq 0$ on $\partial \Omega$. The multiplication operator $u \mapsto \varphi u$ is called vanishing of order $m$ on $\partial \Omega$ if the function $\varphi$ is vanishing of order $m$ on $\partial \Omega$.
\label{def:simply_vanishing_equivalent}
\end{definition}

Functions that are vanishing of order 1 are simply vanishing functions. To see why, take $\Omega\subset\R^d$ with $\mathscr{C}^1$ boundary and assume $\varphi \in \mathscr{C}^k(\bar{\Omega})$ is simply vanishing. Fix any $y \in \partial \Omega$ and let $\Omega_n = B(y,n^{-1}) \cap \Omega$. Since $\partial \Omega$ is $\mathscr{C}^1$, there exists a point $x_n \in \Omega_n$ such that $|x_n - y| = \text{dist}(x_n,\partial \Omega)$. Given $\varphi$ is simply vanishing, there exists an $a \neq 0$ such that
\begin{align*}
a = \lim_{n \to\infty} \frac{\varphi(x_n)}{\text{dist}(x_n,\partial \Omega)} 
=\lim_{n \to\infty} \frac{|\varphi(x_n) - \varphi(y)|}{|x_n - y|} = |\nabla \varphi(y)|,
\end{align*}
which shows $\nabla \varphi(y) \neq 0$. Since $y \in \partial \Omega$ was arbitrary, we see that $\nabla \varphi \neq 0$ on $\partial \Omega$ whenever $\partial \Omega$ is $\mathscr{C}^1$.

We have a similar definition for matrix-valued functions. Let $\Phi : \bar{\Omega} \to \C^{d \times d}$ be Hermitian for each $x \in \bar{\Omega}$. Then there exists a unitary matrix $\mathbf{U}(x)$ and a real diagonal matrix $\mathbf{D}(x)$ such that 
\begin{align}\label{eq:schur's decomposition}
\Phi = \mathbf{U} \mathbf{D} \mathbf{U}^*,
\end{align}
by Schur's decomposition theorem. If $\Phi \in \mathscr{C}^1(\bar{\Omega}; \C^{d\times d})$, we can choose $\mathbf{U}$ and $\mathbf{D}$ in $\mathscr{C}^{1}(\bar{\Omega}; \C^{d\times d})$ so that \eqref{eq:schur's decomposition} holds. Thus, we lose no generality by assuming the operator $\mathbf{D}$ has the form $\mathbf{D} = \text{diag}(\varphi_1, \ldots, \varphi_d)$ for some functions $\varphi_i \in \mathscr{C}^1(\bar{\Omega})$ where $i = 1,\ldots, d$.

\begin{definition}
Let $\Omega \subset \R^d$ be open and bounded, and let $\Phi \in \mathscr{C}^{m}(\bar{\Omega}; \C^{d \times d})$. We say the function $\Phi$ is \emph{vanishing of order} $m$ if $\Phi$ is positive semi-definite on $\bar{\Omega}$, and the matrix $\mathbf{D} = \text{diag}(\varphi_1,\ldots,\varphi_d)$ in its Schur decomposition has the following property: for each $i = 1,\ldots,d$, either $\varphi_i > 0$ on $\bar{\Omega}$ or $\varphi_i$ is vanishing of order $m_i$ on $\partial \Omega$ with $m_i \leq m$, with at least one function $\varphi_i$ that is vanishing of order $m$ on $\partial \Omega$. The multiplication operator $\mathbf{u} \mapsto \Phi \mathbf{u}$ is called vanishing of order $m$ if the function $\Phi$ is vanishing of order $m$. 
\label{def:simply vanishing matrices2}
\end{definition}

\subsection{Fredholm and semi-Fredholm operators}\label{subsection:fredholm}

We will utilize Fredholm operator theory when describing the properties of the multiplication operators $u \mapsto \varphi u$ and $\mathbf{u} \mapsto \Phi \mathbf{u}$. In this section, we briefly review the theory of Fredholm operators. 

Let $X$ and $Y$ be Banach spaces. An operator $A : X \to Y$ is called \emph{Fredholm} if
\begin{enumerate}[label=(\alph*)]
\item The domain of $A$ is dense in $X$,
\item The operator $A$ is closed on its domain,
\item The nullspace of $A$ is finite dimensional, \label{nullspace_finite_dimensional}
\item The range of $A$ is closed in $Y$,
\item The co-dimension of the range of $A$ is finite dimensional, \label{codimension_of_the_range}
\end{enumerate}
where the co-dimension of a closed subspace $M \subset Y$, denoted $\co\dim M$, is the dimension of the quotient space $Y /M$. We use $\mathcal{F}(X,Y)$ to denote the set of Fredholm operators from $X$ to $Y$ and write $\mathcal{F}(X)$ in place of $\mathcal{F}(X,X)$. Note that property \ref{codimension_of_the_range} is equivalent to requiring that the nullspace of the adjoint operator $A^*$ be finite dimensional. The \emph{index} of a Fredholm operator $A$ is defined as
\begin{align*}
\ind(A) \coloneqq \dim \nul(A) - \co\dim \range(A).
\end{align*}

The set of \emph{semi-Fredholm} operators from $X$ to $Y$, denoted $\mathcal{F}_+(X,Y)$, is the set of operators that satisfy all the properties of Fredholm opertors except possibly property \ref{codimension_of_the_range}. The set of semi-Fredholm operators from $X$ to $X$ will be denoted by $\mathcal{F}_+(X)$. We note that our definition for $\mathcal{F}_+(X,Y)$ is sometimes referred to as the set of \emph{upper semi-Fredholm} operators. 

The following characterization of Fredholm operators is useful. 

\begin{thm}[\cite{schechter2002principles}, Theorem 7.1]
\label{thm:equivalence_of_fredholm}
Let $X$ and $Y$ be Banach spaces. Then $A \in \mathcal{F}(X,Y)$ if and only if there exists closed subspaces $X_0 \subset X$ and $Y_0 \subset Y$ where $Y_0$ is finite dimensional and $X_0$ has finite co-dimension such that
\begin{align*}
X = X_0 \oplus \nul(A), \qquad Y = \range(A) \oplus Y_0.
\end{align*}
Moreover, there exists operators $A_0 \in \mathcal{B}(Y,X)$, $K_1 \in \mathcal{B}(X)$, and $K_2 \in \mathcal{B}(Y)$ where
\begin{itemize}
\item The $\nul(A_0) = Y_0$,
\item The $\range(A_0) = X_0 \cap \domain(A)$,
\item $A_0 A = I - K_1$ on $\domain(A)$,
\item $A A_0 = I - K_2$ on $Y$,
\item The $\nul(K_1) = X_0$, while $K_1 = I$ on $\nul(A)$,
\item The $\nul(K_2) = \range(A)$, while $K_2 = I$ on $Y_0$.
\end{itemize}
\end{thm}

Note that $K_1$ and $K_2$ are projection operators and their ranges are finite dimensional. The operator $A_0$ from Theorem~\ref{thm:equivalence_of_fredholm} will be referred to as the \emph{pseudo-inverse} of $A$, since $A A_0 A = A$ and $A_0 A A_0 = A_0$.

An equivalent characterization of semi-Fredholm operators is as follows: if $X$ and $Y$ are Banach spaces and $A \in \mathcal{C}(X,Y)$, then $A$ is not semi-Fredholm if and only if there exists a bounded sequence $\{x_k\} \subset \domain(A)$ having no convergent subsequence such that $\{Ax_k\}$ converges. A proof of this equivalence can be found in \cite{schechter1976remarks} or \cite[p. 177]{schechter2002principles}.



\begin{rem}
Suppose $\Omega\subset \R^d$ is open and bounded, $\varphi \in \mathscr{C}^{m}(\bar{\Omega})$ is simply vanishing on $\partial \Omega$ and $\zeta \in \mathscr{C}^{m}(\bar{\Omega})$ is vanishing of order $m$ on $\partial \Omega$. Then the multiplication operator $u \mapsto \varphi^m u$ is semi-Fredholm from $L^p(\Omega)$ to $W^{k,p}(\Omega)$ if and only if the mapping $u \mapsto \zeta u$ is semi-Fredholm. Moreover, $\domain(\varphi^m) = \domain(\zeta)$. To see why, use the fact that the multiplication operators $u \mapsto \zeta \varphi^{-m} u$ and $u \mapsto \varphi^{m} \zeta^{-1} u$ are one-to-one and onto $L^p(\Omega)$ and that the composition of a semi-Fredholm operator with a Fredholm operator is semi-Fredholm. 
\end{rem}

\section{Basic properties of closed operators}\label{section:basics}

Fredholm operators are closed under composition. That is, if $X,Y$, and $Z$ are Banach spaces, then $A \in \mathcal{F}(X,Y)$ and $B \in \mathcal{F}(Y,Z)$ implies $BA \in \mathcal{F}(X,Z)$ with $\ind(BA) = \ind(B) + \ind(A)$. Moreover, if $B \in \mathcal{C}(Y,Z)$ and $BA \in \mathcal{F}(X,Z)$ then we necessarily have $B \in \mathcal{F}(Y,Z)$. These claims are proved, respectively, in \cite[p. 157]{schechter2002principles} as Theorem 7.3, and \cite[p. 162]{schechter2002principles} as Theorem 7.12. As for semi-Fredholm operators, we have the following.

\begin{lem}[\cite{schechter1976remarks}, Lemma 4]\label{lem:reverse_products_and_semifredholm}
Let $X,Y$, and $Z$ be Banach spaces and let $A \in \mathcal{F}(X,Y)$, and $B \in \mathcal{C}(Y,Z)$. If $BA \in \mathcal{F}_+(X,Z)$, then $B \in \mathcal{F}_+(Y,Z)$.
\end{lem}




One of the theorems that we use throughout this article is the following consequence of the Closed Graph Theorem. A proof of the Closed Graph Theorem can be found in many functional analysis textbooks, such as \cite[p. 62]{schechter2002principles} or \cite[p. 166]{kato1995perturbation}. 


\begin{lem}\label{lem:equivalence_of_norms}
Let $X$ be a Banach space and $Y \subset X$. If there exists a norm that converts $Y$ into a Banach space then there exists a $c > 0$ such that $\|y\|_X \leq c \|y\|_Y$ for all $y \in Y$. If $Y = X$ then their norms are equivalent.
\end{lem}

\begin{proof}
Let $\iota$ denote the inclusion map from $Y$ to $X$. It is a closed operator with domain equal to $Y$, so by the Closed Graph Theorem it is bounded. If $Y=X$ then apply the above argument to the inclusion map from $X$ to $Y$.
\end{proof}

The above lemma is useful for showing that special subsets of $L^p(\Omega)$ have certain properties --- such as compactness --- since they can inherit such properties from other Sobolev spaces. One of the most important consequences of compactness is the following theorem. 

\begin{thm}\label{thm:compactness equals closed range}
Let $X$ and $Y$ be Banach spaces. If $A \in \mathcal{C}(X,Y)$ with $\domain(A) \subset\subset X$, then $A \in \mathcal{F}_+(X,Y)$.
\end{thm}

\begin{proof}
We prove this by contradiction. Suppose $A$ is not semi-Fredholm. Since $A \in \mathcal{C}(X, Y)$, this implies there exists a bounded sequence $\{x_n \}\subset \domain(A)$ having no convergent subsequence, such that $\{Ax_n \}$ is convergent in $Y$. But if $\{x_n\}$ is bounded in $X$ and $\{A x_n\}$ is convergent in $Y$, then $\{x_n\}$ is a bounded sequence in the $A$-norm. Since $\domain(A) \subset\subset X$ there exists a subsequence of $\{x_n\}$ that is convergent in $X$, the desired contradiction.
\end{proof}

\begin{rem}
We know an operator $A \in \mathcal{C}(X,Y)$ has closed range if and only if there exists a $c > 0$ such that
\begin{align*}
\inf_{z \in \nul(A)} \|x - z\|_{X} \leq c \|Ax\|_{Y}, \qquad \text{for all } x \in \domain(A).
\end{align*}
Theorem~\ref{thm:compactness equals closed range} can be used to quickly establish estimates involving differential operators. We can, for example, establish Poincar\'e inequalities.

It is also well known that $W^{1,p}_{0}(\Omega) \subset\subset L^p(\Omega)$ for any bounded and open set $\Omega \subset \R^d$. Since the weak gradient operator $\nabla$ is closed on $W^{1,p}_0(\Omega)$, we get $\range(\nabla)$ is closed by Theorem~\ref{thm:compactness equals closed range}. This implies the existence of a $c > 0$ such that
\begin{align*}
\|u\|_{L^p(\Omega)} = \inf_{a \in \nul(\nabla)} \|u - a\|_{L^p(\Omega)} \leq c\|\nabla u\|_{L^p(\Omega)},
\end{align*}
for any $u \in W^{1,p}_0(\Omega)$.
\end{rem}

\section{Properties of vanishing operators}\label{section:properties_vanishing}

\subsection{The domain of a vanishing operator}\label{subsection:domain_vanishing}

In this section we establish properties of the domain of the vanishing operator $\varphi : L^p(\Omega) \to W^{k,p}(\Omega)$. In particular, we will establish the embedding of the domain of $\varphi$ in various Sobolev spaces. 

If $\varphi \in \mathscr{C}^1(\bar{\Omega})$, then the mapping $u \mapsto \varphi u$ is bounded from $L^p(\Omega)$ to $L^p(\Omega)$. This implies that a natural choice for its domain is $L^p(\Omega)$. Whenever a vanishing operator $\varphi$ is composed with a differential operator $A$ to form $A\varphi$ --- $A$ being an operator that is closed on $W^{k,p}(\Omega)$ --- it makes sense to think of $\varphi$ as a densely defined operator that maps some subset of $L^p(\Omega)$ to the space $W^{k,p}(\Omega)$.

This and subsequent sections rely heavily on Hardy's inequality, so we include the statement for the reader's convenience.

\begin{thm}[\cite{wannebo1990hardy}, Hardy's Inequality]\label{thm:hardys_inequality}
Let $\Omega \subset \R^d$ be a bounded open set with $\mathscr{C}^{0,1}$ boundary, and $\delta(x) = \inf_{y \in \partial \Omega} |x - y|$. Then, for all $u \in \mathscr{C}^{\infty}_{0}(\Omega)$,
\begin{align*}
\| \delta^{-m} u\|_{L^p(\Omega)}
\leq c \| \nabla^{(m)} u \|_{L^{p}(\Omega)},
\end{align*}
where $c>0$ depends on $\Omega$, $p$, $d$, and $m$.
\end{thm}

See \cites{lehrback2014weighted, hajlasz1999pointwise} for recent developments on the assumptions necessary for Hardy's inequality. The interested reader should consult \cite[\S 2.7]{maz2011sobolev} for a treatment of optimal constants for Hardy's inequality.
 
We begin with basic properties of the domain and range of the multiplication operator $u \mapsto \varphi^m u$.


\begin{lem}\label{lem:statement of range}
Let $\Omega \subset \R^d$ be open and bounded and $\varphi \in \mathscr{C}^1(\bar{\Omega})$ be simply vanishing on $\partial \Omega$. For each $k,m \in \Z_+$, the multiplication operator $\varphi^m : L^p(\Omega) \to W^{k,p}(\Omega)$ defined as $u \mapsto \varphi^m u$ is closed on 
\begin{align*}
\domain(\varphi^m) = \{ u \in L^p(\Omega) : \varphi^m u \in W^{k,p}(\Omega) \}.
\end{align*}
Moreover, if $\varphi \in \mathscr{C}^m(\bar{\Omega})$ and $m \leq k$ then $\range(\varphi^m) = W^{k,p}(\Omega) \cap W^{m,p}_{0}(\Omega)$.
\end{lem}

\begin{proof}
The proof is broken into two claims.

\begin{claim}{1}
The multiplication operator $u \mapsto \varphi^m u$ is closed on $\domain(\varphi^m)$.
\end{claim}

We first show it is closable. Suppose $u_n \to 0$ in $L^p(\Omega)$ and $\varphi^m u_n \to y$ in $W^{k,p}(\Omega)$. Then we know that $\varphi^m u_n \to y$ in $L^p(\Omega)$. But since $\varphi$ is bounded and $u_n \to 0$ in $L^p(\Omega)$ we can conclude $\varphi^m u_n \to 0 = y$ in $L^p(\Omega)$. This shows the multiplication operator $\varphi^m$ is closable on its domain. But any closed extension cannot be defined on a set larger than $\domain(\varphi^m)$, implying the domain of any closed extension must be $\domain(\varphi^m)$. This completes the proof of the claim.

\begin{claim}{2}
If $\varphi \in \mathscr{C}^m(\bar{\Omega})$ and $m \leq k$ then $\range(\varphi^m) = W^{k,p}(\Omega) \cap W^{m,p}_{0}(\Omega)$.
\end{claim}

Take $v \in W^{k,p}(\Omega) \cap W^{m,p}_{0}(\Omega)$. Since $v \in W^{m,p}_0(\Omega)$ we can apply Hardy's inequality (Theorem~\ref{thm:hardys_inequality}) to show that $\varphi^{-m} v \in L^p(\Omega)$. Since this implies $\varphi^{-m} v \in \domain(\varphi^m)$ we have $W^{k,p}(\Omega) \cap W^{m,p}_{0}(\Omega) \subset \range(\varphi^m)$.

For the other direction, first note that since $\varphi^m \in \mathscr{C}^m(\bar{\Omega}) \cap W^{m,p}_0(\Omega)$ there exists a sequence $\{\phi_n\} \subset \mathscr{C}^{\infty}_0(\Omega)$ such that $\phi_n \to \varphi^m$ in $W^{m,p}(\Omega)$ and $\{D^{\alpha}\phi_n\}$ is uniformly bounded when $|\alpha| \leq m$\footnote{One such example are the functions $\phi_{n} = \varphi^m \mathbf{1}_{\Omega_{n}} * \eta_{1 / 3n}$, where $\mathbf{1}_{\Omega_n}$ is an indicator function for the set $\Omega_{n}= \{x \in \Omega : \text{dist}(x, \partial \Omega) > 1/n \}$ and $\eta_{\epsilon}$ is a mollifier.}. Let $u \in \domain(\varphi^m)$ be arbitrary, and set $v_n = \phi_n - \varphi^m$. Since $D^{\alpha}(\phi_n u)$ is bounded by $cD^{\alpha}(\varphi u)$ for some constant $c$ and $|D^{\alpha}(v_n u)|^p \to 0$ almost everywhere when $|\alpha| \leq m$ we have
\begin{align*}
\lim_{n\to\infty} \|D^{\alpha} (\phi_n u) - D^{\alpha} (\varphi^m u) \|_{L^p(\Omega)}
=\lim_{n\to\infty} \|D^{\alpha}  (v_n u ) \|_{L^p(\Omega)} = 0,
\end{align*}
when $|\alpha| \leq m$ by dominated convergence. Noting that $\phi_n u \in W^{m,p}_0(\Omega)$ shows $\varphi^m u \in W^{m,p}_0(\Omega)$ and completes the proof.
\end{proof}

The following lemma establishes the relative compactness of $\domain(\varphi^m)$ in $L^p(\Omega)$ when $\varphi^m$ maps to $W^{k,p}(\Omega)$ for $k > m$. It uses the relative compactness of $W^{m+1,p}(\Omega)$ in $W^{m,p}(\Omega)$. This is implied by the Rellich-Kondrachov Theorem, which establishes that for $1 \leq p < \infty$, $W^{1,p}(\Omega)$ is compactly embedded in $L^p(\Omega)$ whenever $\Omega$ is a bounded domain with Lipshitz continuous boundary. See \cite[p. 168]{adams2003sobolev} Thoerem 6.3 for the full statement and proof of the Rellich-Kondrachov Theorem.

\begin{lem}\label{lem:compactness implies compactness almost}
Let $\Omega \subset \R^d$ be open and bounded with $\mathscr{C}^{0,1}$ boundary and let $\varphi \in \mathscr{C}^{1}(\bar{\Omega})$ be simply vanishing on $\partial \Omega$. If
\begin{align*}
\domain(\varphi^m) = \{u \in L^p(\Omega) : \varphi^m u \in W^{m+1,p}(\Omega) \}
\end{align*}
then $\domain(\varphi^m) \subset\subset L^p(\Omega)$.
\end{lem}

\begin{proof}
Since $\Omega$ is bounded and $\partial \Omega$ is $\mathscr{C}^{0,1}$, we can use the Rellich-Kondrachov Theorem to establish $W^{m+1,p}(\Omega) \subset \subset W^{m,p}(\Omega)$. Suppose $\{ u_n \} \subset \domain(\varphi^m)$ is such that $\|u_n\|_{\domain(\varphi^m)} \leq 1$ for each $n$. Then $\| \varphi^m u_n\|_{W^{m+1,p}(\Omega)} \leq 1$ for each $n$, so there exists a subsequence that is convergent in $W^{m,p}(\Omega)$. After relabeling the convergent subsequence, we take this to be the entire sequence. Applying Hardy's inequality (Theorem~\ref{thm:hardys_inequality}) yields
\begin{align*}
\lim_{n,k\to \infty} \|u_n - u_k\|_{L^{p}(\Omega)} 
    \leq \lim_{n,k\to \infty}  c \| \varphi^m u_n - \varphi^m u_k\|_{W^{m,p}(\Omega)} = 0,
\end{align*}
completing the proof.
\end{proof}

Next we establish the embedding of $\domain(\varphi)$ in various Sobolev spaces. To do so, we will use the fact that when $\Omega \subset \R^d$ is open and bounded with $\mathscr{C}^k$ boundary, the map 
\[
    u \mapsto u |_{\partial \Omega}
\]
from $\mathscr{C}^k(\bar{\Omega})$ to $\mathscr{C}^{k}(\partial \Omega)$ can be extended to a continuous surjective linear map from $W^{k,p}(\Omega)$ to $W^{k-1/p,p}(\partial \Omega)$ where $1 < p < \infty$ (see \cite[p. 158]{demengel2012functional} Theorem 3.79). 


We would like to highlight the fact that the trace map $T$ on $W^{k,p}(\Omega)$ is defined on a Banach space and has range that is onto the Banach space $W^{k-1/p,p}(\partial \Omega)$. As for the nullspace of $T$, a classical result states that when $\partial \Omega$ is $\mathscr{C}^1$,  $Tu = 0$ if and only if $u \in W^{1,p}_0(\Omega)$; (see \cite[p. 259]{evans1998partial} Theorem 2, or \cite[p. 138]{demengel2012functional} Corollary 3.46). 





Let $\Omega$ be an open and bounded set with $\mathscr{C}^k$ boundary and let $T$ denote the continuous trace operator from $W^{k,p}(\Omega)$ onto $W^{k-1/p,p}(\partial \Omega)$. We will need a trace-like operator that is one-to-one. To define this new operator, first set $W_0 \coloneqq W^{1,p}_0(\Omega) \cap W^{k,p}(\Omega) = \nul(T)$. Clearly $W_0$ is a closed subspace of $W^{k,p}(\Omega)$. Next let $\hat{W}^k$ denote the quotient space $W^{k,p}(\Omega)/W_0$ and define the operator $\hat{T}$ from $\hat{W}^k$ to $W^{k-1/p,p}(\partial \Omega)$ as
\begin{align*}
\hat{T}[u] = Tu, \qquad [u] \in \hat{W}^k.
\end{align*}
The operator $\hat{T}$ is well-defined, linear, and one-to-one. To see that $\hat{T}$ is closed, take $\{[u_n]\} \subset \hat{W}^k$ such that $[u_n] \to [u]$ and $\hat{T}[u_n] \to y$ as $n\to \infty$. Then $\| [u_n] - [u]\|_{\hat{W}^k} \to 0$ implies the existence of a sequence $\{v_n\} \subset W_0$ such that
\begin{align*}
u_n - v_n \to u, \qquad \text{in } W^{k,p}(\Omega).
\end{align*}
Since $\hat{T}[u_n] \to y$, we know that $T u_n = T(u_n - v_n) \to y$ as $n \to \infty$. By the boundedness of $T$ we get $T u = y$. This implies $\hat{T}[u] = y$ and proves that $\hat{T}$ is closed. Applying the Closed Graph Theorem shows that $\hat{T}$ is bounded. Moreover, the fact that $T$ is surjective implies that $\hat{T}$ is surjective as well. This tells us that $\hat{T}^{-1}$ exists and is a bounded linear operator from $W^{k-1/p,p}(\partial \Omega)$ onto $\hat{W}^k$, by the Bounded Inverse Theorem. 



We also need the following general Sobolev space theorem. We use $\lfloor a \rfloor$ to denote the integer part of $a$. 

\begin{thm}[\cite{adams2003sobolev} p. 85,  Sobolev Imbedding]\label{thm:general_inequality}
Let $\Omega \subset \R^d$ be open and bounded with a $\mathscr{C}^{0,1}$ boundary. Assume $u \in W^{k,p}(\Omega)$, and that $kp> d$. Set 
\begin{align*}
\kappa = k - \left\lfloor \tfrac{d}{p} \right\rfloor - 1, \qquad
\gamma = \left\{
\begin{array}{ll}
1 - \left( \tfrac{d}{p} - \left\lfloor \tfrac{d}{p} \right\rfloor\right), & \textnormal{if } \tfrac{d}{p} \not \in \N \\
\textnormal{any number in } (0,1), & \textnormal{otherwise.}
\end{array}\right.
\end{align*}
Then there exists a function $u^*$ such that $u^* = u$ a.e. and $u^* \in \mathscr{C}^{\kappa,\gamma}(\bar{\Omega})$.
\end{thm}

We can now establish the following lemma.

\begin{lem}\label{lem:kinda hardy}
Let $\Omega \subset \R^d$ be a bounded open set and let $\varphi \in \mathscr{C}^{1}(\bar{\Omega})$ be simply vanishing on $\partial \Omega$. Assume $k \in \N$ with $kp > {d}$ and that the boundary $\partial \Omega$ is $\mathscr{C}^k$. Then for every $u \in L^{p}(\Omega)$ where $\varphi u \in W^{k+1,p}(\Omega)$ we have
\begin{align*}
\varphi D^{\alpha} u \in W^{1,p}_{0}(\Omega) \cap W^{k,p}(\Omega), \qquad 
    \text{for any } \alpha \text{ with } |\alpha| = 1.
\end{align*}
\end{lem}

\begin{proof}
The proof is divided into two claims.

\begin{claim}{1}
If $|\alpha| = 1$ then $D^{\alpha} (\varphi u)  = u D^{\alpha}\varphi$ on $\partial \Omega$.
\end{claim}

Since $kp > d$ and $\varphi u \in W^{k+1,p}(\Omega)$, we know there exists a $\gamma \in (0,1)$ dependent on $d$ and $p$ such that
\begin{align*}
\varphi u \in \mathscr{C}^{k-\big\lfloor \tfrac{d}{p} \big\rfloor,\gamma}(\bar{\Omega}), 
\end{align*}
by Theorem~\ref{thm:general_inequality}. This shows $\varphi u \in \mathscr{C}^{1}(\bar{\Omega})$. Also, since $u = \varphi^{-1} \varphi u$ whenever $\varphi \neq 0$, the fact that $\varphi \in \mathscr{C}^{1}(\bar{\Omega})$ and is nonzero in $\Omega$ implies $u \in \mathscr{C}^1(\Omega)$.

Now, since $\varphi$ is simply vanishing on $\partial \Omega$, we know that for any $y \in \partial \Omega$,
\begin{align*}
\lim_{x \to y} \frac{|x-y|}{|\varphi(x)|} 
=   \lim_{x \to y} \frac{|x-y|}{|\varphi(x)-\varphi(y)|} = |\nabla \varphi(y)|^{-1},
\end{align*}
where the limit is taken in $\Omega$. Note that $\nabla \varphi \neq 0$ on $\partial \Omega$, so that this is well defined. Next, given $\varphi u \in W^{1,p}(\Omega)$ and $u \in L^p(\Omega)$, we know $\varphi u \in W^{1,p}_{0}(\Omega)$ by Lemma~\ref{lem:statement of range}. We already know $\varphi u$ is continuous on $\bar{\Omega}$, which implies $\varphi u = 0$ on $\partial \Omega$. Thus, for any $y \in \partial \Omega$ we can take any sequence in $\Omega$ that converges to $y$ and obtain
\begin{align}
\lim_{x \to y} |u(x)|
=   \lim_{x \to y} 
        \frac{|\varphi(x) u(x) - \varphi(y)u(y)|}{|x - y|} 
        \frac{|x-y|}{|\varphi(x)|} 
=   |\nabla (\varphi u)(y)| |\nabla \varphi(y)|^{-1}.
\label{eq:limit_with_product}
\end{align}
By Leibniz's rule, $D^{\alpha} (\varphi u) = uD^{\alpha} \varphi + \varphi D^{\alpha} u$ when $|\alpha| = 1$, so the above limit implies $D^{\alpha} (\varphi u)  = u D^{\alpha}\varphi$ on $\partial \Omega$. 

\begin{claim}{2}
The function $\varphi D^{\alpha}u$ is in $W^{1,p}_{0}(\Omega) \cap W^{k,p}(\Omega)$ whenever $|\alpha| = 1$. 
\end{claim}

Set ${W}_0 \coloneqq W^{1,p}_{0}(\Omega) \cap W^{k,p}(\Omega)$ and $\hat{W}^k \coloneqq W^{k,p}(\Omega)/W_0$. By assumption, $\varphi u \in W^{k+1,p}(\Omega)$ so $D^{\alpha}(\varphi u) \in W^{k,p}(\Omega)$ whenever $|\alpha| = 1$. This implies the coset $[D^{\alpha}(\varphi u)]$ is in $\hat{W}^k$ and that its trace $\hat{T}[D^{\alpha}(\varphi u)]$ is in $W^{k-1/p,p}(\partial \Omega)$.
By claim 1,
\begin{align*}
u D^{\alpha} \varphi = D^{\alpha}(\varphi u) \qquad \text{on } \partial \Omega,
\end{align*}
which shows that $u D^{\alpha} \varphi |_{\partial \Omega} \in W^{k-1/p,p}(\partial \Omega)$ and that  $\hat{T}^{-1} (u D^{\alpha} \varphi |_{\partial \Omega}) \in \hat{W}^k$.  The one-to-one nature of $\hat{T}^{-1}$ implies $[D^{\alpha}(\varphi u)] = [uD^{\alpha}\varphi]$. Since these cosets are equal, there exists a function $v \in W_0$ such that
\begin{align*}
u D^{\alpha} \varphi = D^{\alpha} (\varphi u) - v.
\end{align*}
But again, $D^{\alpha} (\varphi u) = uD^{\alpha} \varphi + \varphi D^{\alpha} u$, so it must be the case that $v = \varphi D^{\alpha} u$. We then conclude that $\varphi D^{\alpha} u \in W_0 =  W^{1,p}_{0}(\Omega) \cap W^{k,p}(\Omega)$, completing the proof.
\end{proof}

\begin{thm}\label{thm:kinda hardy_pre}
Let $\Omega \subset \R^d$ be a bounded open set and let $\varphi \in \mathscr{C}^{1}(\bar{\Omega})$ be simply vanishing on $\partial \Omega $. Assume $k \in \N$ with $kp > {d}$ and that $\partial \Omega$ is $\mathscr{C}^k$. Then $u \in L^{p}(\Omega)$ and $\varphi u \in W^{k+1,p}(\Omega)$ implies $u \in W^{\kappa,p}(\Omega)$ where $\kappa \coloneqq k- \big\lfloor \tfrac{d}{p} \big\rfloor$.
\end{thm}

\begin{proof}
Fix the multi-index $\alpha$ with $|\alpha| \leq \kappa$. Choose a finite sequence of multi-indices $\{\alpha_n\}_{n\leq |\alpha|}$ each with $|\alpha_n| = 1$ such that $\sum \alpha_n = \alpha$. By assumption, $kp > d$ so we may apply Lemma~\ref{lem:kinda hardy} to $u$ to obtain
\begin{align}\label{consequence of lemma kinda hardy}
\varphi D^{\alpha_1} u \in W^{1,p}_{0}(\Omega) \cap W^{k+1-|\alpha_1|,p}(\Omega).
\end{align}
Applying Hardy's inequality to $\varphi D^{\alpha_1} u$ yields
\begin{align}\label{using hardys after lemma}
D^{\alpha_1} u \in L^{p}(\Omega).
\end{align}
Moreover, given $|\alpha| \leq \kappa = k - \big\lfloor \tfrac{d}{p} \big\rfloor$, we know that
\begin{align}\label{index inequality kinda hardy}
k+1 - |\alpha_1| \geq k+1 - |\alpha| \geq 1 + \big\lfloor \tfrac{d}{p} \big\rfloor > \tfrac{d}{p}.
\end{align}
Since \eqref{consequence of lemma kinda hardy}, \eqref{using hardys after lemma}, and \eqref{index inequality kinda hardy} all hold, we may apply Lemma~\ref{lem:kinda hardy} to $D^{\alpha_1}u$ to obtain $\varphi D^{\alpha_1+\alpha_2} u \in W^{1,p}_{0}(\Omega) \cap W^{k-1,p}(\Omega)$. Another application of Hardy's inequality shows $D^{\alpha_1 + \alpha_2} u \in L^{p}(\Omega)$. We continue inductively applying Lemma~\ref{lem:kinda hardy} and Hardy's inequality at each step to finally show that
\begin{align*}
\varphi D^{\alpha} u \in W^{1,p}_{0}(\Omega) \cap W^{k+1-|\alpha|,p}(\Omega), \quad \text{and} \quad D^{\alpha}u \in L^p(\Omega).
\end{align*}
Since this applies to any multi-index $\alpha$ with $|\alpha| \leq \kappa$, we see that $u \in W^{\kappa,p}(\Omega)$, completing the proof.
\end{proof}

Iteratively applying the above theorem yields the following.

\begin{thm}\label{thm:kinda hardy}
Let $\Omega \subset \R^d$ be a bounded open set and let $\varphi \in \mathscr{C}^1(\bar{\Omega})$ be simply vanishing on $\partial \Omega$. Fix $m \in \N$ and $1 < p < \infty$. Assume $k \in \N$ is such that 
\begin{align}\label{the lower bound for k kinda hardy}
k > \tfrac{d}{p} + (m-1)\big\lfloor \tfrac{d}{p} \big\rfloor,
\end{align}
and that $\partial \Omega$ is $\mathscr{C}^k$. If $u \in L^{p}(\Omega)$ and $\varphi^m u \in W^{k+m,p}(\Omega)$ then
\begin{align}\label{statement of kappa kinda hardy}
u \in W^{\kappa,p}(\Omega), \quad \text{where } \kappa \coloneqq k- m\big\lfloor \tfrac{d}{p} \big\rfloor.
\end{align}
\end{thm}

\begin{proof}
For convenience, we define the variables $\kappa_1, \ldots, \kappa_m$ as follows
\begin{align*}
\kappa_j \coloneqq k + m - j - j \big\lfloor \tfrac{d}{p} \big\rfloor, \qquad j = 1, \ldots, m.
\end{align*}
We know $\varphi^{m-1} u \in L^p(\Omega)$, $\varphi^m u \in W^{k+m,p}(\Omega)$, and $k + m - 1 > p^{-1} d$. By Theorem~\ref{thm:kinda hardy_pre}, this implies $\varphi^{m-1} u \in W^{\kappa_1,p}(\Omega)$. If $m > 1$, we see that $\kappa_1 - 1 > p^{-1}d$, and we have $\varphi^{m-2} u \in L^p(\Omega)$ and $\varphi^{m-1} u \in W^{\kappa_1,p}(\Omega)$. Thus we get $\varphi^{m-2} u \in W^{\kappa_2,p}(\Omega)$ by the same theorem. Continuing inductively, we apply Theorem~\ref{thm:kinda hardy_pre} at each step to get $\varphi^{m-j} u \in W^{\kappa_j,p}(\Omega)$ for $m > j$. When $j = m - 1$ we have $u \in L^p(\Omega)$ and $\varphi u \in W^{\kappa_{m-1},p}(\Omega)$. Since 
\begin{align*}
\kappa_{m-1} -1 = k - (m-1)\big\lfloor \tfrac{d}{p} \big\rfloor > \tfrac{d}{p},
\end{align*}
we may apply Theorem~\ref{thm:kinda hardy_pre} one more time to get $u \in W^{\kappa_m,p}(\Omega)$, as desired.
\end{proof}

\begin{rem}\label{rem:implicit bound}
There is an implicit estimate accompanying Theorem~\ref{thm:kinda hardy}. Assume $\varphi \in \mathscr{C}^1(\bar{\Omega})$ and consider the multiplication operator $\varphi^m : L^p(\Omega) \to W^{k+m,p}(\Omega)$ where $k$ satisfies \eqref{the lower bound for k kinda hardy}. Theorem~\ref{thm:kinda hardy} tells us that $u \in \domain(\varphi^m)$ implies $u \in W^{\kappa,p}(\Omega)$ where $\kappa$ is given by \eqref{statement of kappa kinda hardy}. Thus, $\domain(\varphi^m) \subset W^{\kappa,p}(\Omega)$. By the closedness of $\varphi^m$, $\domain(\varphi^m)$ is a Banach space with the operator norm. We can conclude that, for some constants $c_0, c_1 > 0$ and for all $u \in \domain(\varphi^m)$,
\begin{align}
\|u\|_{W^{\kappa,p}(\Omega)} & \leq c_0 \|u\|_{\domain(\varphi^m)}
 = c_0 \|u\|_{L^{p}(\Omega)} + c_0 \|\varphi^m u\|_{W^{k+m,p}(\Omega)} \label{first inequality from closed graph} \\
& \leq c_1 \|\varphi^m u\|_{W^{k+m,p}(\Omega)}, \label{second inequality from hardys inequality}
\end{align}
where \eqref{first inequality from closed graph} follows from Lemma~\ref{lem:equivalence_of_norms} applied to the Banach spaces $\domain(\varphi^m)$ and $W^{\kappa, p}(\Omega)$ and \eqref{second inequality from hardys inequality} from Hardy's inequality.
\end{rem}

\subsection{The range of a vanishing operator}\label{subsec:range vanishing}

Having a closed range is a very useful property for linear operators. As we will see in section~\ref{section:spectrum}, it is often necessary for establishing basic properties of the spectrum. Showing the multiplication operator $u \mapsto \varphi u$ has closed range requires keeping track of the multiplicity of the roots of the function $\varphi$. This is formally established in the following result.



\begin{thm}\label{thm:range closed vanishing}
Let $\Omega \subset \R^d$ be open and bounded with $\mathscr{C}^{0,1}$ boundary and assume the function $\varphi \in \mathscr{C}^k(\bar{\Omega})$ is simply vanishing on $\partial \Omega$. Then the range of the operator $u \mapsto \varphi^m u$ is closed in $W^{k,p}(\Omega)$ whenever $k \geq m$ and is not closed when $k < m$.
\end{thm}

\begin{proof}
As usual, we treat $\varphi$ as an operator from some dense subset of $L^p(\Omega)$ to $W^{k,p}(\Omega)$. We start with the following.

\begin{claim}{1}
If $k \geq m$ then the range of $\varphi^m$ is closed in $W^{k,p}(\Omega)$.
\end{claim} 

If $k = m$ then $\range(\varphi^m) = W^{m,p}_{0}(\Omega)$ as discussed in Lemma~\ref{lem:statement of range}, which clearly establishes the closedness of $\range(\varphi^m)$ in $W^{m,p}(\Omega)$. If $k > m$ then we may apply Lemma~\ref{lem:compactness implies compactness almost} to show $\domain(\varphi^m) \subset\subset L^p(\Omega)$. Invoking Theorem~\ref{thm:compactness equals closed range} proves $\varphi^m$ is semi-Fredholm, which implies $\range(\varphi^m)$ is closed in $W^{k,p}(\Omega)$.

\begin{claim}{2}
If $k < m$ then the range of $\varphi^m$ is not closed in $W^{k,p}(\Omega)$.
\end{claim} 

We prove this claim by contradiction. Suppose $\varphi^m$ has closed range in $W^{k,p}(\Omega)$. This implies $\varphi^m$ is semi-Fredholm from $L^p(\Omega)$ to $W^{k,p}(\Omega)$. Since $\varphi^k$ is Fredholm from $L^p(\Omega)$ to $W^{k,p}_{0}(\Omega)$, and since $\varphi^{m} = \varphi^{m-k} \varphi^k$ is semi-Fredholm from $L^p(\Omega)$ to $W^{k,p}(\Omega)$, Lemma~\ref{lem:reverse_products_and_semifredholm} implies $\varphi^{m-k}$ is semi-Fredholm from $W^{k,p}_{0}(\Omega)$ to $W^{k,p}(\Omega)$. Now, for any function $u \in \mathscr{C}^{\infty}_{0}(\Omega)$ we know that $v = \varphi^{k-m} u \in \mathscr{C}^k_{0}(\Omega)$, so $\varphi^{m-k}v \in \mathscr{C}^{\infty}_{0}(\Omega)$, implying $\varphi^{m-k}$ is onto the subspace $\mathscr{C}^{\infty}_{0}(\Omega)$. This implies 
\begin{align*}
\mathscr{C}^{\infty}_{0}(\Omega) \subset \range(\varphi^{m-k}).
\end{align*}
Since $\range(\varphi^{m-k})$ is closed, we know that $W^{k,p}_{0}(\Omega)$ is a subspace of $\range(\varphi^{m-k})$. But we also know that $\varphi^{k} \in W^{k,p}_{0}(\Omega)$, so there exists a function $v \in W^{k,p}_{0}(\Omega)$ such that $\varphi^{m-k} v = \varphi^{k}$, which implies $v = \varphi^{2k-m}$. But $\varphi^{2k-m}$ cannot be in $W^{k,p}_{0}(\Omega)$ as Hardy's inequality would then show
\begin{align*}
\|\varphi^{k-m}\|_{L^p(\Omega)} = \|\varphi^{-k} \varphi^{2k-m} \|_{L^p(\Omega)} \leq c \|  \varphi^{2k-m} \|_{W^{k,p}(\Omega)} < \infty.
\end{align*}
This is our desired contradiction.
\end{proof}

\begin{example}[The Legendre differential equation]
Set $\Omega = (-1,1)$. Let us analyze the operator $L$ given by
\begin{align*}
Lu(x) = \frac{d}{dx}\left[ (1-x^2) \frac{d}{dx} u(x) \right],
\end{align*}
acting on $L^p(\Omega)$, where as usual we assume $1 < p <\infty$. Let $A$ denote the derivative operator on $L^p(\Omega)$ and $\varphi(x) = (1 - x^2)$. The domain of $A$ is $W^{1,p}(\Omega)$, the nullspace of $A$ is $\text{span}\{ 1\}$, and the range of $A$ is equal to $L^p(\Omega)$. Since $\varphi$ is simply vanishing on $\partial \Omega$, we know the range of the multiplication operator $u \mapsto \varphi u$ is equal to $W^{1,p}_0(\Omega)$ by Lemma~\ref{lem:statement of range}. 

If $u \in W^{1,p}(\Omega)$, then $u \in \mathscr{C}^{0,1/p}(\bar{\Omega})$ by Sobolev Imbedding (Theorem~\ref{thm:general_inequality}). Thus, we can find a unique line $l(x)$ such that $u + l = 0$ on $\partial \Omega$, implying $u + l \in W^{1,p}_0(\Omega)$. Since $u \in W^{1,p}(\Omega)$ was arbitrary, this implies
\begin{align*}
W^{1,p}(\Omega) = W^{1,p}_0(\Omega) \oplus \text{span}\{1,x\}.
\end{align*}
If we take $\tilde{A}$ to be the restriction of the derivative operator to $W^{1,p}_0(\Omega)$, then $\dim\nul(\tilde{A}) = 0$ and $\co\dim\range(\tilde{A}) = 1$. Since $A$, $\tilde{A}$, and $\varphi : L^p(\Omega) \to W^{1,p}_0(\Omega)$ are all Fredholm we know that $L = \tilde{A}\varphi A$ is Fredholm, with
\begin{align*}
\ind(L) = \ind(\tilde{A}\varphi A) = \ind(\tilde{A}) + \ind(\varphi) + \ind( A) = -1 + 0 + 1 = 0,
\end{align*}
and $\nul(L) = \text{span}\{1\}$. 

In terms of the domain of $L$, we automatically get $\domain(L) \subset \domain(A) = W^{1,p}(\Omega)$. The interesting thing to note is that $L$ cannot be semi-Fredholm if $\domain(L) \subseteq W^{2,p}(\Omega)$. To see this, first note that $A$ maps $W^{2,p}(\Omega)$ onto $W^{1,p}(\Omega)$ and that the range of $\varphi : W^{1,p}(\Omega) \to W^{1,p}_0(\Omega)$ is not closed. Since this implies that $\varphi A : W^{2,p}(\Omega) \to W^{1,p}_0(\Omega)$ cannot be semi-Fredholm, we know that $L = \tilde{A}\varphi A$ cannot be semi-Fredholm.
\end{example}

\subsection{Matrix-valued functions} \label{subsection:matrix_functions}

One of our goals is to aid in the analysis of 
\begin{align}\label{L_operator2}
Lu = \divergence( \Phi \nabla u),
\end{align}
when the matrix-valued function $\Phi : \bar{\Omega} \to \C^{d \times d}$ is positive semi-definite for each $x\in \bar{\Omega}$. With that end in mind, this section focuses on the multiplication operator $\mathbf{u} \mapsto \Phi \mathbf{u}$ where $\Phi \in \mathscr{C}^1(\bar{\Omega};\C^{d\times d})$ and $\mathbf{u}(x) \in \C^d$ for almost every $x \in \Omega$. As we will see shortly, the properties that were established for the multiplication operator $u \mapsto \varphi u$ apply for the multiplication operator $\mathbf{u} \mapsto \Phi \mathbf{u}$ as well.

In order for the operator $L$ defined in \eqref{L_operator2} to be uniformly elliptic, the matrix $\Phi : \bar{\Omega} \to \C^{d\times d}$ must be uniformly positive definite. This section, like the ones before it, focus on the violation of this positivity assumption. Specifically, we assume $\Phi$ is vanishing of order $m$ (recall Definition~\ref{def:simply vanishing matrices2}). Another way to express this is as follows: for each fixed $V \subset\subset \Omega$ we have
\begin{align} \label{rayleigh thingy}
\inf_{x \in V^{}} \inf_{v \in \C^d} {\bar{v} \cdot \Phi(x) v} \geq c_V|v|,
\end{align}
where $c_V > 0$ is the smallest eigenvalue of $\Phi(x)$ for $x \in V$. Moreover, the speed at which $c_V$ goes to zero is proportional to $a^m$ where $a = \inf_{y\in \partial V} \text{dist}(y,\partial \Omega)$.

We take $L^p(\Omega^d)$ to be the space of all measurable functions $\mathbf{u} = (u_1, \ldots, u_d)$ such that ${u}_i \in L^p(\Omega)$ for $i = 1,\ldots,d$. The norm of $L^p(\Omega^d)$ is taken to be
 \begin{align*}
 \|\mathbf{u} \|_{L^p(\Omega^d)} \coloneqq \sum_{i=1}^d \| u_i\|_{L^p(\Omega)}.
 \end{align*}
In other words,
\begin{align*}
L^p(\Omega^d) = \underbrace{L^p(\Omega) \times  \cdots \times L^p(\Omega).}_{d \text{ copies}}
\end{align*}
The space $W^{k,p}(\Omega^d)$ is defined as the subset of $\mathbf{u} = (u_1,\ldots,u_d) \in L^p(\Omega^d)$ where $u_i \in W^{k,p}(\Omega)$ for each $i = 1,\ldots, d$. We assume $\Phi : L^p(\Omega^d) \to W^{k,p}(\Omega^d)$ for some $k \in \Z_+$.

\begin{thm}\label{thm:range closed vanishing_matrices}
Let $\Omega \subset \R^d$ be open and bounded with $\mathscr{C}^{0,1}$ boundary and assume $\Phi \in \mathscr{C}^k(\bar{\Omega};\C^{d\times d})$ is vanishing of order $m$. Then the range of the operator $\mathbf{u} \mapsto \Phi \mathbf{u}$ is closed in $W^{k,p}(\Omega^d)$ whenever $k \geq m$ and is not closed when $k < m$.
\end{thm}

\begin{proof}
We know there exists $\mathbf{U} \in \mathscr{C}^k(\bar{\Omega}; \C^{d\times d})$ and $\mathbf{D} \in \mathscr{C}^k(\bar{\Omega}; \R^{d\times d})$ such that $\Phi = \mathbf{U}\mathbf{D}\mathbf{U}^*$, where $\mathbf{D} = \text{diag}(\varphi_1, \ldots, \varphi_d)$ and $\varphi_i \in \mathscr{C}^{k}(\bar{\Omega})$ for each $i = 1,\ldots,d$. Since $\mathbf{U}$ is one-to-one and onto $L^p(\Omega^d)$, it suffices to prove the claim for the operator $\mathbf{D}$. Now, by our definition of $W^{k,p}(\Omega^d)$, it must be the case that $\range(\mathbf{D})$ is closed in $W^{k,p}(\Omega^d)$ if and only if the multiplication operators $u \mapsto \varphi_i u$ have closed range in $W^{k,p}(\Omega)$ for each $i = 1,\ldots, d$. With this in mind, we simply apply Theorem~\ref{thm:range closed vanishing} for each diagonal function $\varphi_i$, yielding the desired conclusion.
\end{proof}

\begin{thm}\label{thm:kinda hardy_matrices}
Let $\Omega \subset \R^d$ be a bounded open set and let $\Phi \in \mathscr{C}^k(\bar{\Omega};\C^{d\times d})$ be vanishing of order $m$. Assume $k \in \N$ is such that 
\begin{align*}
k > \tfrac{d}{p} + (m-1)\big\lfloor \tfrac{d}{p} \big\rfloor,
\end{align*}
and that the boundary $\partial \Omega$ is $\mathscr{C}^k$. If $\mathbf{u} \in L^{p}(\Omega^d)$ and $\Phi \mathbf{u} \in W^{k+m,p}(\Omega^d)$ then
\begin{align}
\label{defines kappa}
\mathbf{u} \in W^{\kappa,p}(\Omega^d), \quad \text{where } \kappa \coloneqq k- m\big\lfloor \tfrac{d}{p} \big\rfloor,
\end{align}
and there exists a $c > 0$, independent of $\mathbf{u}$, such that
\begin{align}\label{est:domain estimate matrices}
\|\mathbf{u}\|_{W^{\kappa,p}(\Omega^d)} \leq c \| \Phi \mathbf{u} \|_{W^{k+m,p}(\Omega^d)}.
\end{align}
\end{thm}

\begin{proof}
We know there exists $\mathbf{U} \in \mathscr{C}^k(\bar{\Omega}; \C^{d\times d})$ and $\mathbf{D} \in \mathscr{C}^k(\bar{\Omega}; \R^{d\times d})$ such that $\Phi = \mathbf{U}\mathbf{D}\mathbf{U}^*$, where $\mathbf{D} = \text{diag}(\varphi_1, \ldots, \varphi_d)$ and $\varphi_i \in \mathscr{C}^{k}(\bar{\Omega})$ for each $i = 1,\ldots,d$. As in the above theorem, it suffices to prove the claim for the operator $\mathbf{D}$. 

Given  $\mathbf{u} = (u_1,\ldots, u_d) \in L^{p}(\Omega^d)$ and $\mathbf{D} \mathbf{u} \in W^{k+m,p}(\Omega^d)$ then we have $u_i \in L^p(\Omega)$ and $\varphi_i u_i \in W^{k+m,p}(\Omega)$ for each $i = 1,\ldots,d$. If $\varphi_i > 0$ on $\bar{\Omega}$ then $u_i \in W^{k+m,p}(\Omega)$, and if $\varphi_i$ is vanishing of order $j \leq m$ we apply Theorem~\ref{thm:kinda hardy} to get 
\begin{align*}
u_i \in W^{\kappa_j,p}(\Omega), \quad \text{ where } \kappa_j \coloneqq k - j\big\lfloor \tfrac{d}{p} \big\rfloor.
\end{align*}
In either case, $u_i \in W^{\kappa,p}(\Omega)$ for each $i = 1,\ldots,d$. The proof of inequality \eqref{est:domain estimate matrices} mirrors that of Remark~\ref{rem:implicit bound} and is omitted.
\end{proof}

\section{Differential operators composed with vanishing operators}\label{section:differential_with}

In this section we examine differential operators that are composed with vanishing operators. By `differential operator' we mean any operator that is closed on the subspace $W^{k,p}(\Omega) \subset L^p(\Omega)$, $k \geq 1$, and maps to either $L^p(\Omega)$ or $L^p(\Omega^d)$.  We pay particular attention to linear differential operators that are Fredholm or semi-Fredholm. Many of these results use compactness of nested Sobolev spaces.

\subsection{Compactness}\label{subsection:compactness}

One of the salient features of the Sobolev space $W^{k,p}(\Omega)$ is its compactness relationship with the ambient space $L^p(\Omega)$. In this section, we explore the implications of compactness on the composition of differential operators with vanishing functions.  


We start with the following general result for Fredholm operators.

\begin{thm}
\label{thm:fredholm_on_compactly_embedded_domains}
Let $X$ and $Y$ be Banach spaces. If $A \in \mathcal{F}(X,Y)$ then $\domain(A) \subset\subset X$ if and only if its pseudo-inverse is compact from $Y$ to $X$.
\end{thm}

\begin{proof}
Since $A$ is closed we may equip $\domain(A)$ with the $A$-norm and convert it into a Banach space, which we call $W$. 

\begin{claim}{1}
If $A \in \mathcal{F}(X,Y)$ with $W \subset\subset X$ then the pseudo-inverse of $A$ is compact from $Y$ to $X$.
\end{claim}

Given that $A$ is Fredholm from $X$ to $Y$, we know it is also Fredholm from $W$ to $Y$. Let $\tilde{A}_0$ denote the pseudo-inverse of $A : W \to Y$, and let $\iota : W \to X$ denote the inclusion map from $W$ to $X$. The assumption that $W \subset\subset X$ tells us that $\iota$ is compact, which implies $\iota \tilde{A}_0 : Y \to X$ is compact as well since $\tilde{A}_0 \in \mathcal{B}(Y,W)$. If we let $A_0$ denote the pseudo-inverse of $A : X \to Y$ then we see that $A_0 = \iota \tilde{A}_0$, so $A_0$ is compact.

\begin{claim}{2}
If $A \in \mathcal{F}(X,Y)$ and the pseudo-inverse of $A$ is compact from $Y$ to $X$ then $W \subset\subset X$.
\end{claim}

We are told $A$ is Fredholm, so we know  $\nul(A)$ is finite dimensional and that there exists a closed subspace $X_0 \subset X$ such that $X = X_0 \oplus \nul(A)$. Suppose $\{x_n\} \subset \domain(A)$ with $\|x_n\|_{\domain(A)} \leq c$. Then for each $n$ we have the decomposition $x_n = a_n + b_n$ where $a_n \in X_0$ and $b_n \in \nul(A)$. Since $\|a_n\|_{\domain(A)} \leq c$ for all $n$, $\{A a_n\}$ is a bounded sequence in $Y$. Given that $A_0$, the pseudo-inverse of $A$, is compact from $Y$ to $X$, there exists a subsequence of $\{a_n\} = \{A_0 A a_n\}$ that is convergent in $X$. Also, since $\{b_n\}$ is bounded and $\nul(A)$ is finite dimensional, every subsequence of $\{ b_n\}$ has a further subsequence that is convergent. Thus, we can find a subsequence of $\{b_n\}$ along the convergent subsequence of $\{a_n\}$ that is convergent. With this we can conclude that $\{x_n\} = \{a_n + b_n\}$ contains a convergent subsequence in $X$. This proves the claim and completes the proof of the theorem.
\end{proof}

As a consequence of Theorem~\ref{thm:fredholm_on_compactly_embedded_domains} we have the following.

\begin{thm}
\label{thm:not_closed_product}
Let $X$, $Y$, and $Z$ be Banach spaces. Suppose $A \in \mathcal{F}(X,Y)$ where $\domain(A)\subset\subset X$. If $B \in \mathcal{C}(Y,Z)$ but is not semi-Fredholm then $BA$ is not closed on its natural domain.
\end{thm}

\begin{proof}
The proof is by contradiction. Assume $BA$ is closed on its natural domain,
\begin{align*}
\domain(BA) = \{ x \in \domain(A) : Ax \in \domain(B) \}.
\end{align*}

\begin{claim}{1}
There exists a $c > 0$ such that $\|x\|_{\domain(A)} \leq c\|x\|_{\domain(BA)}$ holds for all $x \in \domain(BA)$. 
\end{claim}

If $BA$ was closed on $\domain(BA)$ then $\domain(BA)$ would be a Banach space with the $BA$-norm. Since $A$ is Fredholm it must be closed on its domain $\domain(A)$, so $\domain(A)$ is also a Banach space with the $A$-norm. We know that $\domain(BA) \subset \domain(A)$, so by Lemma~\ref{lem:equivalence_of_norms} there exists a $c > 0$ such that $\|x \|_{\domain(A)} \leq c\|x\|_{\domain(BA)}$ whenever $x \in \domain(BA)$.

\begin{claim}{2}
There exists a sequence that converges in $\domain(BA)$ but does not converge in $\domain(A)$.
\end{claim}

Since $B$ is not semi-Fredholm, there exists a bounded sequence $\{x_n\} \subset \domain(B)$ such that $\{B x_n \}$ converges but $\{x_n\}$ has no convergent subsequence. Given that $A$ is Fredholm we know $A$ has a pseudo-inverse, which we denote by $A_0$.  We then set $y_n = A_0 x_n$ and notice that
\begin{align*}
Ay_n = A A_0 x_n = (I - K) x_n,
\end{align*}
where $K$ is a projection into some finite dimensional subspace of $Y$. Since $\{x_n\}$ has no convergent subsequence and $K$ projects to a finite dimensional subspace, $\{K x_n\}$ is eventually zero. Thus, $\{A y_n\}$ has no convergent subsequence and $\{BAy_n\}$ converges. Since $\domain(A) \subset\subset X$, we know that $A_0$ is compact by Theorem~\ref{thm:fredholm_on_compactly_embedded_domains} so $\{y_n\}$ has a convergent subsequence in $X$ (which, after relabeling, we take to be the entire sequence). Using claim 1, we have
\begin{align*}
\|y_m - y_n\|_{\domain(A)}
& = \|y_m - y_n\|_X + \|Ay_m - Ay_n\|_Y \\
& \leq c \|y_m - y_n\|_{\domain(BA)} = c \|y_m - y_n\|_X + c\|BA y_m - BAy_n\|_Z.
\end{align*}
We have established that $\{BAy_n\}$ and $\{y_n\}$ converge in $Z$ and $X$ respectively, so $\{y_n\}$ is convergent in $\domain(BA)$. But we know that $\{Ay_n\}$ does not converge in $Y$, hence $\{y_n\}$ cannot converge in $\domain(A)$. This is the desired contradiction.
\end{proof}

If $\varphi \in \mathscr{C}^1(\bar{\Omega})$ is simply vanishing, then by Theorem~\ref{thm:range closed vanishing} the range of the multiplication operator $u \mapsto \varphi u$ is not closed in $L^p(\Omega)$. Thus, $\varphi$ cannot be semi-Fredholm from $L^p(\Omega)$ to $L^p(\Omega)$. The above theorem then says $\varphi^m A$ is never closed on its natural domain for any $m > 0$. However, we can partially make up for this loss by showing that $\varphi^m A$ is closable. Before we begin we will need a few more tools from classical functional analysis. 

The adjoint operator of $A$, denoted $A^*$, is a map from the dual $Y^*$ to $X^*$, where
\begin{align}\label{adjoint def}
A^*y^*(x) = y^*(Ax), \qquad \text{for all } x \in \domain(A),
\end{align}
and some $y^* \in Y^*$. The set of appropriate $y^* \in Y^*$ for which \eqref{adjoint def} holds is $\domain(A^*)$. 

The following two results are needed. 


\begin{lem}[\cite{kato1995perturbation} Lemma 131, p. 137]
Let $X$ be a normed vector space. Suppose that a sequence $\{x_n\} \subset X$ is bounded, and $\lim l^*(x_n) = l^*(x)$ for each $l^* \in V$ where $V$ is a dense subset of $X^*$. Then the sequence $\{x_n\}$ converges to $x$ weakly.
\label{lem:weak_uniform_boundedness}
\end{lem}

\begin{thm}[\cite{schechter2002principles}, Theorems 7.35 and 7.36, p. 178]
\label{thm:densely_defined_adjoint}
Let $X$, $Y$, and $Z$ be Banach spaces, and assume that $A \in \mathcal{C}(X,Y)$ where $\range(A)$ is closed in $Y$ with finite co-dimension. Let $B$ be a densely defined operator from $Y$ to $Z$. Then $(BA)^*$ exists, $(BA)^* = A^*B^*$, and both $(BA)^{*}$ and $BA$ are densely defined.
\end{thm}

With the above lemma and theorem, we can conclude the following.

\begin{thm}\label{thm:closable_product}
Let $X$, $Y$, and $Z$ be Banach spaces. Suppose $A \in \mathcal{F}(X,Y)$ and $B$ is a densely defined linear operator from $Y$ to $Z$. Then $BA$ is closable.
\end{thm}

\begin{proof}
Since $A$ is Fredholm, the domains of $BA$ and $(BA)^*$ are both dense, by Theorem~\ref{thm:densely_defined_adjoint}. Suppose we have a sequence $\{x_n\} \subset \domain(BA)$ where $x_n \to 0$ and $B Ax_n \to z$ as $n \to \infty$. Then for each $w^* \in \domain((BA)^*)$,
\begin{align*}
w^* (z) = \lim_{n\to\infty} w^*(B Ax_n) = \lim_{n\to\infty} (BA)^* w^*(x_n) = 0.
\end{align*}
Since $\domain((BA)^*)$ is dense in $X^*$ and $B Ax_n$ is bounded, this implies that $BAx_n \to 0$ weakly as $n\to\infty$, by Lemma~\ref{lem:weak_uniform_boundedness}. We know that weak limits must coincide with strong limits so $z = 0$. Thus $B A$ is closable.
\end{proof}

Theorems~\ref{thm:not_closed_product} and \ref{thm:closable_product} yield the following result.

\begin{thm}\label{thm:not closed but closable}
Let $\Omega \subset \R^d$ be an open and bounded set with $\mathscr{C}^{0,1}$ boundary. Let $\varphi \in \mathscr{C}^{1}(\bar{\Omega})$ be simply vanishing on $\partial \Omega$. 
If $A : L^p(\Omega) \to L^p(\Omega)$ is Fredholm with $\domain(A) \subset W^{1,p}(\Omega)$, then $\varphi^m A$, for $m \geq 1$, is not closed on its natural domain but is closable.
\end{thm}

\begin{proof}
Theorem~\ref{thm:range closed vanishing} tells us that the range of the multiplication operator $u \mapsto \varphi^{m} u$ is not closed in $L^p(\Omega)$, so $\varphi^m$ is not semi-Fredholm from $L^p(\Omega)$ to $L^p(\Omega)$. Since $\Omega$ is bounded with $\mathscr{C}^{0,1}$ boundary, this implies $\domain(A) \subset W^{1,p}(\Omega) \subset\subset L^p(\Omega)$ by the Rellich-Kondrachov Theorem. Theorem~\ref{thm:not_closed_product} then tells us that $\varphi^m A$ is not closed on its natural domain, but by Theorem~\ref{thm:closable_product} $\varphi^m A$ is closable.
\end{proof}

Our next result makes use of the following theorem.

\begin{thm}[\cite{schechter2002principles} Theorem  7.22, p. 170] \label{thm:reflexive adjoint fredholm}
Let $X$ and $Y$ be Banach spaces. If $A \in \mathcal{F}(X,Y)$ and $Y$ is reflexive, then $A^* \in \mathcal{F}(Y^*, X^*)$ and $\ind(A^*) = -\ind(A)$.
\end{thm}

We now have the following:

\begin{thm}\label{thm:reflexive adjoint compactness}
Let $X$ and $Y$ be Banach spaces where $Y$ is also reflexive. If $A \in \mathcal{F}(X,Y)$ with $\domain(A) \subset\subset X$ then $\domain(A^*) \subset\subset Y^*$.
\end{thm}

\begin{proof}
By Theorem~\ref{thm:reflexive adjoint fredholm} we know that $A^* \in \mathcal{F}(Y^*, X^*)$. Also, if $x^* \in \domain(A^*)$ then
\begin{align}\label{adjoint equation pseudo inverse}
A^*_0 A^*x^*(x) = A^*x^*(A_0x) = x^*(AA_0x) = x^*(x - K_2x) = (I - K_2^*)x^*(x).
\end{align}
Thus, from \eqref{adjoint equation pseudo inverse} and Theorems~\ref{thm:densely_defined_adjoint} and \ref{thm:equivalence_of_fredholm} we conclude:
\begin{align*}
A^* A_0^* = (A_0A)^* = I - K_1^*, \qquad A^*_0 A^* = (AA_0)^* = I - K_2^*,
\end{align*}
which implies $A_0^*$ is the pseudo-inverse of $A^*$. Since $A \in \mathcal{F}(X,Y)$, and $\domain(A) \subset\subset X$, we know that the pseudo-inverse $A_0$ is compact from $Y$ to $X$ by Theorem~\ref{thm:fredholm_on_compactly_embedded_domains}. Given $A_0$ is compact from $Y$ to $X$, we know $A_0^*$ is compact from $X^*$ to $Y^*$. If we then apply Theorem~\ref{thm:fredholm_on_compactly_embedded_domains} to $A^*$ we get $\domain(A^*) \subset\subset Y^*$.
\end{proof}

\subsection{The Spectrum} \label{section:spectrum} 

Let $X$ be a Banach space and $A$ be a densely defined linear operator from $X$ to $X$. The resolvent set of $A$, denoted $\rho(A)$, is the set of all $\lambda \in \C$ such that $A - \lambda$ has a bounded inverse. The complement of $\rho(A)$ in $\C$ is called the spectrum of $A$, and is denoted as $\sigma(A)$. We let $\sigmap(A)$ denote the point spectrum of $A$:
\begin{align*}
\sigmap(A) \coloneqq \{ \lambda \in \C : Ax = \lambda x \text{ for some } x \in \domain(A) \}.
\end{align*}
We define the essential spectrum as:
\begin{align*}
\ess(A) = \bigcap_{K \in \mathcal{K}(X)} \sigma(A + K),
\end{align*}
where $\mathcal{K}(X)$ is the set of all compact operators on $X$. This set is sometimes referred to as \emph{Schechter's essential spectrum}. Another useful characterization of the essential spectrum is given in the theorem below.

\begin{thm}[\cite{schechter2002principles} Theorem 7.27, p. 172]
\label{thm:equivalence_of_schecter_spectrum}
Let $X$ be a Banach space and assume $A \in \mathcal{C}(X)$. Then $\lambda \not\in \ess(A)$ if and only if $A-\lambda \in \mathcal{F}(X)$ and $\ind(A-\lambda) = 0$.
\end{thm}

We will also need the notion of relatively compact operators. If $A \in \mathcal{C}(X,Y)$, an operator $B : X \to Z$ is called \emph{compact relative to $A$} if $B$ is compact from the Banach space $\domain(A)$ to $Z$. The following theorem is a more robust version of the Fredholm Alternative since it is stated for any Fredholm operator $A$ (not just the identity operator) and the perturbations to $A$ can be any operator that is compact relative to $A$. A proof of this theorem can be found in \cite[p. 281]{kato1958perturbation} Theorem 1, or \cite[p. 162]{schechter2002principles} Theorem 7.10.

\begin{thm}[\cite{kato1958perturbation} Theorem 1, p. 281]
\label{thm:fredholm_alternative}
If $A \in \mathcal{F}(X,Y)$ and $B$ is compact relative to $A$ then $A + B \in \mathcal{F}(X,Y)$ and $\ind(A+B) = \ind(A)$.
\end{thm}

\begin{rem}\label{about the essential spectrum}
With the help of Theorems \ref{thm:equivalence_of_schecter_spectrum} and \ref{thm:fredholm_alternative}, we see that the essential spectrum is invariant under relatively compact perturbations. Given that the identity map on $X$ is compact relative to $A \in \mathcal{C}(X)$ whenever $\domain(A) \subset\subset X$, we know that either $\ess(A) = \emptyset$ or $\ess(A) = \C$. This fact makes calculating the essential spectrum of differential operators relatively easy whenever we can use the Rellich-Kondrachov Theorem.
\end{rem}

\begin{thm}
\label{thm:abstract_ess_reverse_product}
Let $X$ be a Banach space, $A \in \mathcal{C}(X)$, and  $B \in \mathcal{B}(X)$ be a one-to-one operator where $\domain(A) \not \subset \range(B)$. If $\domain(AB)$ is dense in $X$ with $\domain(AB) \subset\subset X$ and $\rho(A)$ is nonempty then $\ess(AB) =\C$.
\end{thm}

\begin{proof}
The proof is by contradiction. Suppose $\ess(AB) \neq \C$. As mentioned in Remark~\ref{about the essential spectrum}, $\ess(AB) \neq \C$ implies $\ess(AB) = \emptyset$ since $\domain(AB) \subset\subset X$. This implies $AB \in \mathcal{F}(X)$ and $\ind(AB) = 0$ by Theorem~\ref{thm:equivalence_of_schecter_spectrum}. 

Since $\domain(AB) \subset\subset X$, any bounded operator on $X$ is compact relative to $AB$. In particular, $B$ is compact relative to $AB$. By Theorem~\ref{thm:fredholm_alternative}, Fredholm operators and their indices are invariant under relatively compact perturbations. Thus, for any $\eta \in \C$ we have $AB - \eta B \in \mathcal{F}(X)$ and 
\begin{align*}
	\ind(AB) = \ind(AB- \eta B) = 0.
\end{align*}
Now, if $\eta \in \rho(A)$ then $\nul(A-\eta) = \{0\}$ and $\range(A-\eta) = X$, implying $A- \eta$ maps $\domain(A)$ to $X$. But since $B$ does not map to all of $\domain(A)$ we get $\range((A-\eta)B) \neq X$. In other words
\begin{align*}
\text{co-\!} \dim \range (A B - \eta B) > 0.
\end{align*}
Recall that $\nul(B) = \{0\}$ by assumption, and $\nul(A-\eta) = \{0\}$ when $\eta \in \rho(A)$, which gives $\nul\big( (A - \eta) B \big) = \{0\}$. Thus,
\begin{align*}
\ind(AB) 
=     \ind(AB -\eta B)
=     \dim \nul(AB - \eta B) - \text{co-\!}\dim \range(AB - \eta B) < 0,
\end{align*}
which contradicts the fact that $\ind(AB) = 0$.
\end{proof}

The proof of the above theorem yields the following corollary.

\begin{cor}
\label{cor:index_reverse_product}
Let $X$ be a Banach space, $A \in \mathcal{C}(X)$, and $B \in \mathcal{B}(X)$ be a one-to-one operator where the range of $B$ is not the entirety of $\domain(A)$. If $AB$ is Fredholm with $\domain(AB) \subset\subset X$ and $\rho(A)$ is nonempty then $\ind(AB) < 0$.
\end{cor}

In some cases, we are concerned with the adjoint operator $\varphi^m A^*$ instead $A\varphi^m$. For example, one might be interested in the spectral properties of the operator $L$ given by $Lu = -(1-x^2) u_{xx}$ on $\Omega = (-1, 1)$. This operator is the adjoint of $A\varphi u = -( (1-x^2)u )_{xx}$ where $A$ is the Laplacian on $L^2(\Omega)$ and $\varphi(x) = (1-x^2)$. In this case, we have the following theorem.

\begin{thm}
\label{thm:roots equal spectrum}
Let $\Omega \subset \R^d$ be open and bounded with $\mathscr{C}^{0,1}$ boundary, $m,k \in \N$ with $m < k$, and let $A$ be densely defined on $L^p(\Omega)$. Assume
\begin{enumerate}[label=(\alph*),ref=(\alph*)]
\item $A$ is closed on $L^p(\Omega)$ and $\domain(A) \subset W^{k,p}(\Omega)$.
\item There exists a $u \in \domain(A)$ such that $u \notin W^{m,p}_0(\Omega) \cap W^{k,p}(\Omega)$. \label{item:range not everything}
\item The resolvent set $\rho(A)$ is non-empty.
\end{enumerate}
If $\varphi \in \mathscr{C}^k(\bar{\Omega})$ is simply vanishing then
\begin{align*}
\ess(A \varphi^m ) = \ess(\overline{\varphi^m A^*} ) =\C. 
\end{align*}
Moreover, if either $A\varphi^m$ or $\overline{\varphi^m A^*}$ is Fredholm then
\begin{align*}
\sigmap(\overline{\varphi^m A^*} ) = \C.
\end{align*}
\end{thm}

\begin{rem}
By Theorem~\ref{thm:not closed but closable}, ${\varphi^m A^*}$ is never closed on its natural domain when $A$ is Fredholm. Thus, we examine its closure since statements about the essential spectrum are uninformative for operators that are not closed.
\end{rem}

\begin{proof}
As usual, let $\varphi^m$ denote the multiplication operator $u \mapsto \varphi^m u$ from $L^p(\Omega)$ to $W^{k,p}(\Omega)$. The proof is broken into 4 claims.

\begin{claim}{1}
$A\varphi^m$ is closed on its natural domain and $\domain(A\varphi^m) \subset\subset L^p(\Omega)$.
\end{claim}

Given $\partial \Omega$ is $\mathscr{C}^{0,1}$, we see that $\domain(A) \subset W^{k,p}(\Omega) \subset\subset L^p(\Omega)$ by the Rellich-Kondrachov Theorem. Since $A$ is closed on $\domain(A)$, $A$ must be semi-Fredholm on $L^p(\Omega)$ by Theorem~\ref{thm:compactness equals closed range}. With the assumption that $m < k$, Lemma~\ref{lem:compactness implies compactness almost} implies $\domain(\varphi^m) \subset\subset L^p(\Omega)$, and applying Theorem~\ref{thm:compactness equals closed range} yields $\varphi^m$ is semi-Fredholm from $L^p(\Omega)$ to $W^{k,p}(\Omega)$. Since both $A$ and $\varphi^m$ are semi-Fredholm, $A\varphi^m$ is closed on its natural domain. The fact that
\begin{align*}
\domain(A\varphi^m) \subset \domain(\varphi^m) \subset\subset L^p(\Omega)
\end{align*}
completes the proof of the claim.

\begin{claim}{2}
$\ess(A\varphi^m) = \C$
\end{claim}

Applying Lemma~\ref{lem:statement of range} to $\varphi^m$ shows that $\range(\varphi^m) = W^{m,p}_0(\Omega) \cap W^{k,p}(\Omega)$. This and assumption \ref{item:range not everything} implies $\varphi^m$ does not map to all of $\domain(A)$. Given claim 1, $\rho(A)$ is non-empty, and the multiplication operator $\varphi^m : L^p(\Omega) \to L^p(\Omega)$ is bounded, we may apply Theorem~\ref{thm:abstract_ess_reverse_product} to prove $\ess(A\varphi^m) = \C$.

\begin{claim}{3}
If $A\varphi^m$ is not Fredholm then $\ess(\overline{\varphi^m A^*} ) =\C$.
\end{claim}

Fix $\lambda \in \mathbb{C}$. Claim 1 implies the identity is compact relative to $A\varphi^m$. Since $A\varphi^m$ is not Fredholm, $A\varphi^m - \lambda$ cannot be Fredholm by Theorem~\ref{thm:fredholm_alternative}. By Theorem~\ref{thm:reflexive adjoint fredholm}, this implies $\overline{\varphi^m A^*} - \bar{\lambda}$ is not Fredholm. Applying Theorem~\ref{thm:equivalence_of_schecter_spectrum} to $\overline{\varphi^m A^*} - \bar{\lambda}$ shows $\bar{\lambda} \in \ess(\overline{\varphi^m A^*})$. Noting that $\lambda \in \C$ was arbitrary completes the proof of the claim.

\begin{claim}{4}
If $A\varphi^m$ or $\overline{\varphi^m A^*}$ is Fredholm then $\ess(\overline{\varphi^m A^*} ) = \sigmap(\overline{\varphi^m A^*} ) = \C$.
\end{claim}

Given $L^p(\Omega)$ is reflexive, $A\varphi^m$ is Fredholm if and only if $\overline{\varphi^m A^*}$ is Fredholm by Theorem~\ref{thm:reflexive adjoint fredholm}. Since $\rho(A)$ is non-empty and $\domain(A\varphi^m) \subset\subset L^p(\Omega)$, Corollary~\ref{cor:index_reverse_product} then implies $\ind(A\varphi^m) < 0$. Thus,
\begin{align*}
\ind(\overline{\varphi^m A^*}) = -\ind(A \varphi^m ) > 0.
\end{align*}
Moreover, by Theorem~\ref{thm:reflexive adjoint compactness} and claim 1, we have that
\begin{align*}
\domain(\overline{\varphi^m A^*}) \subset\subset L^p(\Omega),
\end{align*}
so for any $\lambda \in \C$, we can conclude
\begin{align} \label{index by positive}
\ind( \overline{\varphi^m A^*} - \lambda) = \ind( \overline{\varphi^m A^*}) > 0.
\end{align}
From \eqref{index by positive} we have $\dim \nul(\overline{\varphi^m A^*} - \lambda) > 0$, which implies $\lambda \in \sigmap(\overline{\varphi^m A^*})$. By Theorem~\ref{thm:equivalence_of_schecter_spectrum}, \eqref{index by positive} also implies $\lambda \in \ess( \overline{\varphi^m A^*} )$. Since $\lambda \in \C$ was arbitrary, we are done.
\end{proof}

\section{Acknowledgements}

Parts of this paper have grown out of work that I completed in my thesis. I would like to thank my thesis advisor, Jay Douglas Wright, for his teachings and innumerable
insightful discussions.

\bibliographystyle{amsplain}
\bibliography{wave_refs,books,analysis}


\end{document}